\pdfoutput=1
\documentclass[12pt, a4paper]{article}
\usepackage[top=0.7in,bottom=1.35in,left=1.35in, right=1.35in]{geometry}
\usepackage[svgnames]{xcolor}
\usepackage{setspace}
\usepackage{abstract}

\usepackage[font=footnotesize,labelfont=bf]{caption}
\usepackage{subcaption}

\usepackage{mathtools}
\usepackage[abbrev,lite,nobysame]{amsrefs}
\usepackage{amsthm,amsfonts}

\usepackage{csquotes}

\usepackage{tikz}
\usepgflibrary{patterns}
\usetikzlibrary{patterns.meta,decorations.pathmorphing
	,decorations.text,positioning,shapes.geometric} 

\usepackage{pgfplots}
\pgfplotsset{compat = newest}
\usepgfplotslibrary{colormaps}

\usepackage[absolute,overlay]{textpos}
\tikzstyle{blue0} = [rectangle, rounded corners, minimum width=2.2cm, minimum height=1.4cm,text centered, draw=black, fill=CornflowerBlue!10, text width=3.2cm]
\tikzstyle{blue} = [rectangle, rounded corners, minimum width=3.6cm, minimum height=1.4cm,text centered, draw=black, fill=CornflowerBlue!10, text width=5.5cm]
\tikzstyle{bluea} = [rectangle, rounded corners, minimum width=2cm, minimum height=1.4cm,text centered, draw=black, fill=CornflowerBlue!10, text width=2.5cm]
\tikzstyle{blueb} = [rectangle, rounded corners, minimum width=3.5cm, minimum height=1.4cm,text centered, draw=black, fill=CornflowerBlue!10, text width=4cm]
\tikzstyle{bluec} = [rectangle, rounded corners, minimum width=5.6cm, minimum height=1.4cm,text centered, draw=black, fill=CornflowerBlue!10, text width=5.8cm]
\tikzstyle{arrow} = [thick,->]

\usepackage{enumerate}

\usepackage[colorlinks=true,linkcolor=NavyBlue,
citecolor=NavyBlue,urlcolor=NavyBlue,bookmarksdepth=3]{hyperref}

\usepackage{cleveref}
\crefname{equation}{}{}

\usepackage[T1]{fontenc}
\usepackage[looser]{newtxtext}
\usepackage[bigdelims]{newtxmath}
\usepackage{microtype}
\usepackage{bm}
\AtBeginDocument{\mathcode`v=\varv}

\numberwithin{equation}{section}

\theoremstyle{plain}
\newtheorem{thrm}{Theorem}[section]
\newtheorem{lmm}[thrm]{Lemma}

\theoremstyle{definition}
\newtheorem{dfntn}[thrm]{Definition}
\newtheorem{xmpl}[thrm]{Example}
\newtheorem{rmrk}[thrm]{Remark}

\theoremstyle{plain}

\makeatletter
\newcommand*{\toccontents}{\@starttoc{toc}}
\makeatother

\newcommand{\xvec}[1]{\bm{#1}}
\newcommand{\xsym}[1]{\bm{#1}}
\newcommand{\xdop}[1]{\bm{\mathrm{#1}}}
\def\xnab{\xdop{\nabla}}

\newcommand{\xwcurl}[1]{\xdop{\nabla}\wedge{#1}}

\newcommand{\xdiv}[1]{\xdop{\nabla}\cdot{#1}}

\newcommand{\xdx}[1]{{{\rm d}#1}}
\def\xdrv#1#2{\frac{{\rm d}#1}{{\rm d}#2}}%

\def\cf{\emph{cf.\/}}\def\eg{\emph{e.g.\/}}

\def\etal{\emph{et al.\/}}

\def\@Rref#1{\hbox{\rm \ref{#1}}}
\def\Rref#1{\@Rref{#1}}

\def\xCzero{{\rm C}^{0}}
\def\xCone{{\rm C}^{1}} 
 
\def\xCinfty{{\rm C}^{\infty}} 

\def\xH{{\rm H}}

\def\xHn#1{{\rm H}^#1}

\def\bVn#1{\mathbf{V}^#1}

\def\xWn#1{{\rm W}^#1}

\def\xLtwo{{\rm L}^{2}}

\def\xU{{\rm U}}

\def\xS{{\rm S}}

\def\xB{{\rm B}}
\def\xR{{\rm R}}
\def\xN{{\rm N}}
\def\xJ{{\rm J}}
\def\xD{{\rm D}}

\begin{document}\linespread{1.05}\selectfont
	\date{}
	
	\author{Vahagn~Nersesyan\,\protect\footnote{NYU-ECNU Institute of Mathematical Sciences at NYU Shanghai, 3663 Zhongshan Road North, Shanghai, 200062, China, e-mail: \href{mailto:vahagn.nersesyan@nyu.edu}{Vahagn.Nersesyan@nyu.edu}}\and
		Manuel~Rissel\,\protect\footnote{NYU-ECNU Institute of Mathematical Sciences at NYU Shanghai, 3663 Zhongshan Road North, Shanghai, 200062, China, e-mail: \href{mailto:Manuel.Rissel@nyu.edu}{Manuel.Rissel@nyu.edu}}}
	
	\title{Localized and degenerate controls for the incompressible Navier--Stokes system}

	\maketitle
	
	\begin{abstract}
		We consider the global approximate controllability of the two-dimensional incompressible Navier--Stokes system driven by a physically localized and degenerate force. In other words, the fluid is regulated via four scalar controls that depend only on time and appear as coefficients in an  effectively constructed driving force supported in a given subdomain. Our idea consists of squeezing low mode controls into a small region, essentially by tracking their actions along the characteristic curves of a linearized vorticity equation. In this way, through explicit constructions and by connecting Coron's return method with recent concepts from geometric control, the original problem for the nonlinear Navier--Stokes system is reduced to one for a linear transport equation steered by a global force. This article can be viewed as an attempt to tackle a well-known open problem due to Agrachev.
	\end{abstract}
	
	\vspace{10pt}\noindent\textbf{Keywords:} approximate controllability, degenerate controls, localized controls, Navier--Stokes equations, return method, transported Fourier modes
	\newline

	\setcounter{tocdepth}{1}
	\toccontents \vspace{4pt}
	\normalsize
	\newgeometry{margin=1.35in}
	
	\section{Introduction}\label{section:introduction}
	The study of fluid mechanics has ever since been intertwined with aspirations to not only observe, but also to manipulate or regulate flows of liquids. Modern control theory, to this end, offers various notions for the formalization of such problems. The question of controllability, in particular, is to ask whether and in which sense external influences (the controls) can cause a system to transition between designated states. From a mathematical stance, and due to the nonlinear nature of the involved constituents, several types of difficulties may arise. These are often related to the size of the prescribed data, how the controls are allowed to enter the system, and the overall degree of degeneracy in the formulated controllability problem. In this sense, the present article sets out to accommodate three requirements. First, the admissible states might be located far away from each other. Second, the fluid should be acted upon merely in a subdomain of small area. Third, it is desired that the driving force can be expressed through explicit formulas in terms of a fixed number of unknown control parameters.

	We establish the global approximate controllability for incompressible viscous fluids on the torus $\mathbb{T}^2 = \mathbb{R}^2/ 2\pi\mathbb{Z}^2$ by means of a degenerate force which is physically localized in a given subdomain $\Omega \subset \mathbb{T}^2$. The keyword \enquote{global} refers to the admissible distance between initial and target profiles being unlimited, the property \enquote{degenerate} emphasizes that there are only few degrees of freedom, while \enquote{localized} means here that the control region~$\Omega$ can be any subdomain containing two cuts $\mathcal{C}_1, \mathcal{C}_2 \subset \Omega$ rendering~$\mathbb{T}^2\setminus (\mathcal{C}_1 \cup \mathcal{C}_2)$ simply-connected (\cf~\Cref{Figure:twocontroldomains}). More precisely, the force employed as a control is supported in the control region~$\Omega$ and explicitly depends, through a formula, on four unknown control parameters. Aside from these four parameters, which are functions only of time, and a scaling constant, the external force shall be fixed independently of the prescribed data. 
	Our efforts contrast the common approach of searching localized interior controls in the form of purely abstract elements of infinite-dimensional function spaces.

	\begin{figure}[ht!]
		\centering
		\resizebox{0.92\textwidth}{!}{
			\begin{subfigure}[b]{0.5\textwidth}
				\centering
				\resizebox{1\textwidth}{!}{
					\begin{tikzpicture}
						\clip(-0.05,-0.05) rectangle (9.05,9.05);

						\fill[line width=0.3mm, fill=CornflowerBlue!80] plot[smooth cycle] (3.4,0) rectangle ++(0.2,9);
						\fill[line width=0.3mm, fill=CornflowerBlue!80] plot[smooth cycle] (0,3.2) rectangle ++(9,0.6);
						
						
						\draw[line width=0.2mm, color=black, dashed, dash pattern=on 8pt off 8pt] plot[smooth cycle] (0,0) rectangle ++(9,9);
						
					\end{tikzpicture}
				}
				\label{Figure:controldomain1}
			\end{subfigure}
			\hfill \quad \hfill
			\begin{subfigure}[b]{0.5\textwidth}
				\centering
				\resizebox{1\textwidth}{!}{
					\begin{tikzpicture}
						\clip(-0.05,-0.05) rectangle (9.05,9.05);
						
						\draw[line width=1.2mm, color=CornflowerBlue!80] plot [smooth] coordinates {(1,0) (1,0.5) (1, 1.5) (1.5, 2) (5,1) (6,3) (6.5,3.5) (7,4) (8,5)  (7,6)  (5,8) (2,4.5) (1,5.5) (1,7) (1,9)};
						
						\draw[line width=3.5mm, color=CornflowerBlue!80] plot [smooth] coordinates {(0,3.2) (1,3.2) (2, 3.2) (4, 4) (5, 4) (7,3.2) (8,3.2) (9,3.2)};
						
						
						\draw[line width=0.2mm, color=black, dashed, dash pattern=on 8pt off 8pt]	plot[smooth cycle] (0,0) rectangle ++(9,9);
						
					\end{tikzpicture}
				}
				\label{Figure:controldomain2}
			\end{subfigure}
		}
		\caption{Two of the various possibilities for localizing the controls. In each picture, the (blue) filled part illustrates a valid control region $\Omega \subset \mathbb{T}^2$, which can be taken of arbitrary nonzero area.}
		\label{Figure:twocontroldomains}
	\end{figure}
	
	\subsection{The controllability problem} \label{subsection:controllabilityproblem}
	Let a viscous Newtonian fluid of velocity $\xvec{u}\colon\mathbb{T}^2\times(0,T_{\operatorname{ctrl}}) \longrightarrow \mathbb{R}^2$ and exerted pressure $p \colon\mathbb{T}^2\times(0,T_{\operatorname{ctrl}}) \longrightarrow \mathbb{R}$ be governed by the two-dimensional incompressible Navier--Stokes system
	\begin{equation}\label{equation:NavierStokesVelocity}
		\partial_t \xvec{u} - \nu \Delta \xvec{u} + \left(\xvec{u} \cdot \xnab\right) \xvec{u} + \xnab p = \xvec{f} + \mathbb{I}_{\Omega} \xsym{\xi}, \quad \xdiv{\xvec{u}} = 0, \quad 
		\xvec{u}(\cdot, 0) = \xvec{u}_0,
	\end{equation}
	where $\nu > 0$ quantifies the viscosity of the fluid at hand, $\xvec{u}_0\colon\mathbb{T}^2 \longrightarrow \mathbb{R}^2$ stands for~the initial velocity, $\xvec{f}\colon\mathbb{T}^2\times(0,T_{\operatorname{ctrl}}) \longrightarrow \mathbb{R}^2$ is a known external force,~$\mathbb{I}_{\Omega}$~denotes the indicator function of the subset $\Omega$,
	and $\xsym{\xi}\colon\mathbb{T}^2\times(0,T_{\operatorname{ctrl}}) \longrightarrow \mathbb{R}^2$ represents a control sought to approximately steer the fluid in any fixed time $T_{\operatorname{ctrl}} > 0$ towards any prescribed admissible target state $\xvec{u}_1\colon\mathbb{T}^2 \longrightarrow \mathbb{R}^2$. The system \eqref{equation:NavierStokesVelocity} is said to be globally approximately controllable if for any given states~$\xvec{u}_0$ and $\xvec{u}_1$ belonging to a certain space with norm $\|\cdot\|$ and arbitrarily chosen accuracy parameter $\varepsilon > 0$, there exists a control $\xsym{\xi}$ such that
	the corresponding solution to~\eqref{equation:NavierStokesVelocity} satisfies
	\[
		\| \xvec{u}(\cdot,T_{\operatorname{ctrl}}) - \xvec{u}_1 \| < \varepsilon.
	\] 
	
	In view of applications where only a limited number of control actions can be realized, and where explicit representations of the controls are required, it would be desirable to achieve the global approximate controllability of \eqref{equation:NavierStokesVelocity} by means of finite-dimensional controls
	\begin{equation}\label{equation:truefinitedimensionalcontrol}
		\xsym{\xi}(\xvec{x},t) = \alpha_1(t) \xsym{\psi}_1(\xvec{x}) + \dots +\alpha_N(t) \xsym{\psi}_N(\xvec{x}),
	\end{equation}
	where $\xsym{\psi}_1,\dots,\xsym{\psi}_N\colon\mathbb{T}^2\longrightarrow \mathbb{R}^2$ are linearly independent vector fields unrelated to the prescribed data and the viscosity. This task amounts to identifying a fixed number of parameters $\alpha_1(t), \dots, \alpha_N(t)$. When $\Omega = \mathbb{T}^2$, the celebrated Agrachev--Sarychev approach \cite{AgrachevSarychev2005} and its various advancements allow solving problems of this type. In reality, however, designing controls to act in the whole domain is difficult to justify. Therefore, the spotlight is put on the challenging case $\Omega \neq \mathbb{T}^2$. In fact, the construction of physically localized and finite-dimensional controls constitutes a widely open problem described by Agrachev in \cite[Section 7]{Agrachev2014}.

	The approach proposed in the present article can be viewed as a step towards this open problem in the following sense. We explicitly fix eight effectively constructed functions
	\[
		\xsym{\vartheta}_1,\dots,\xsym{\vartheta}_4\colon\mathbb{T}^2\times [0, 1] \longrightarrow \mathbb{R}^2, \quad \xsym{\vartheta}_5,\dots,\xsym{\vartheta}_8\colon\mathbb{T}^2 \longrightarrow \mathbb{R}^2, \quad \bigcup_{l=1}^8 \operatorname{supp}(\xsym{\vartheta}_l) \subset \Omega
	\]
	and subsequently seek the control~$\xsym{\xi} \colon \mathbb{T}^2\times(0,T_{\operatorname{ctrl}}) \longrightarrow\mathbb{R}^2$ of the specific form
	\begin{equation}\label{equation:IntroductionLocalizedControl}
		\begin{aligned}
			\xsym{\xi}(\xvec{x},t) = \sum_{l=1}^{4}\gamma_l(t) \xsym{\vartheta}_l(\xvec{x}, \sigma (T_{\operatorname{ctrl}}-t)) + \sum_{l=5}^{8} \gamma_l(t) \xsym{\vartheta}_l(\xvec{x}),
		\end{aligned}
	\end{equation}
	where $\sigma > T_{\operatorname{ctrl}}^{-1}$ and $\gamma_1(t) = \dots = \gamma_{4}(t) = 0$ for $t \leq T_{\operatorname{ctrl}}-\sigma^{-1}$. We call the force~$\xsym{\xi}$ \enquote{finitely decomposable}, as in the explicit decomposition \eqref{equation:IntroductionLocalizedControl} there appear only finitely many unknown parameters $\gamma_1, \dots, \gamma_8$, which are functions only of time, and $\sigma > 0$, which is a large number. For a more precise breakdown how~$\xsym{\xi}$ will be localized with respect to the space variables, we take any fixed nonempty subdomain~$\omegaup \subset \Omega$ and aim to construct $\xsym{\vartheta}_1,\dots,\xsym{\vartheta}_8$ with
	\begin{gather*}
		\operatorname{supp}(\xsym{\vartheta}_1) \cup \dots \cup \operatorname{supp}(\xsym{\vartheta}_4) \subset \omegaup\times(0,1), \quad \operatorname{supp}(\xsym{\vartheta}_5) \cup \dots \cup \operatorname{supp}(\xsym{\vartheta}_6) \subset \omegaup, \\ \operatorname{supp}(\xsym{\vartheta}_7) \cup \operatorname{supp}(\xsym{\vartheta}_8) \subset \Omega.
	\end{gather*}
	Our choice of the profiles~$\xsym{\vartheta}_1,\dots,\xsym{\vartheta}_8$ is universal, only depending on the control region. In particular, they are independent of the prescribed initial state, the target state, the viscosity, the force $\xvec{f}$, and the control time $T_{\operatorname{ctrl}}$. In this way, the global approximate controllability of \eqref{equation:NavierStokesVelocity} shall be established by determining $\gamma_1,\dots,\gamma_8$ and $\sigma \geq T_{\operatorname{ctrl}}^{-1}$, depending on the given data (\cf~\Cref{Figure:ta}). In fact, we are able to express $\gamma_5, \dots, \gamma_8$ by means of $\gamma_1, \dots, \gamma_4$ through universal formulas (\cf~\eqref{equation:zetagam1}), which justifies the point of view that one merely has to act on the system \eqref{equation:NavierStokesVelocity} with four controls: $\gamma_1,\dots,\gamma_4 \in \xLtwo((0,T_{\operatorname{ctrl}}); \mathbb{R})$.

	\quad

	\begin{figure}[ht!]\centering\resizebox{0.7\textwidth}{!}{\centering 
			\begin{tikzpicture}[node distance=1.8cm]
				\node (mmm) [blue0] {Fix the control region $\Omega \subset \mathbb{T}^2$ and $\omegaup \subset \Omega$};
				
				\node (ccc) [bluea, right = 0.8cm of mmm] {Effectively construct $\xsym{\vartheta}_1,\dots,\xsym{\vartheta}_8$};
				
				\node (ppp) [blueb, right = 0.8cm of ccc] {Select $\nu > 0$, $T_{\operatorname{ctrl}} > 0$, $\varepsilon > 0$, $\xvec{f}$, and states $\xvec{u}_0$ and $\xvec{u}_1$};
				
				\node (rrr) [blue, below = 0.8cm of ppp] {Obtain controls $\gamma_1,\dots,\gamma_4$, determine $\gamma_5,\dots,\gamma_8$, and fix $\sigma \geq T_{\operatorname{ctrl}}^{-1}$ sufficiently large};
				
				\node (eee) [bluec, left = 0.8cm of rrr] {The external force $\xsym{\xi}$ defined via \eqref{equation:IntroductionLocalizedControl} steers the fluid from $\xvec{u}_0$ to an $\varepsilon$-neighborhood of $\xvec{u}_1$};

				\draw [arrow,line width=0.6mm] (mmm) -- (ccc);
				\draw [arrow,line width=0.6mm] (ccc) -- (ppp);
				\draw [arrow,line width=0.6mm] (ppp) -- (rrr);
				\draw [arrow,line width=0.6mm] (rrr) -- (eee);
		\end{tikzpicture}}
		\caption{An illustration of the dependencies in the proposed framework.}
		\label{Figure:ta}
	\end{figure}
	
	\subsection{Notations}\label{subsection:notations}
	The basic $\xLtwo$-based spaces of average-free scalar fields and divergence-free vector fields are specified by
	\[
		\xH_{\operatorname{avg}} \coloneq \left\{ f \in 	\xLtwo(\mathbb{T}^2;\mathbb{R}) \, \left| \right. \,  \int_{\mathbb{T}^2} f(\xvec{x}) \, \xdx{\xvec{x}} = 0 \right\}, \,\, \mathbf{H}_{\operatorname{div}} \coloneq \left\{ \xvec{f} \in \xLtwo(\mathbb{T}^2;\mathbb{R}^2) \, \left| \right. \, \xdiv{\xvec{f}} = 0 \right\}.
	\]
	Moreover, for any $m \in \mathbb{N}_0 \coloneq \mathbb{N}\cup\{0\}$, 
	we employ the function spaces
	\begin{equation*}
		\xHn{m} \coloneq \xHn{m}(\mathbb{T}^2;\mathbb{R}) \cap \xH_{\operatorname{avg}}, \quad
		\bVn{m} \coloneq \xHn{m}(\mathbb{T}^2;\mathbb{R}^2) \cap \mathbf{H}_{\operatorname{div}},
	\end{equation*}
	where the $\xLtwo$-based Sobolev spaces $\xHn{m}(\mathbb{T}^2;\mathbb{R})$ and $\xHn{m}(\mathbb{T}^2;\mathbb{R}^2)$ are equipped with the canonical inner product $\langle \cdot, \cdot \rangle_m$ and the induced norm $\|\cdot\|_{m}$. Furthermore, the Lebesgue measure on the torus is assumed normalized such that $\int_{\mathbb{T}^2} \, \xdx{\xvec{x}} = 1$.
	
	\subsection{Main result}\label{subsection:mainresult}
	To formulate our main result, we fix an integer $N \geq 4$ and take the control region~$\Omega$ and its small subset $\omegaup$ from \Cref{subsection:controllabilityproblem}.~Then, as a generalization of \eqref{equation:IntroductionLocalizedControl}, we consider a localized finitely-decomposable force $\xsym{\xi} = \xsym{\xi}_{\gamma_1,\dots,\gamma_N,\sigma}$ of the form
	\begin{equation}\label{equation:IntroductionLocalizedControl2}
		\begin{aligned}
			\xsym{\xi}(\xvec{x},t) = \sum_{l=1}^N\gamma_l(t) \xsym{\vartheta}_l(\xvec{x}, \sigma (T_{\operatorname{ctrl}}-t)) + \sum_{l=N+1}^{N+4}\gamma_{l}(t) \xsym{\vartheta}_{l}(\xvec{x}).
		\end{aligned}
	\end{equation}
	In \eqref{equation:IntroductionLocalizedControl2}, the known effectively constructed profiles are (\cf~\Cref{section:General})
	\[
		\xsym{\vartheta}_1,\dots,\xsym{\vartheta}_N\colon\mathbb{T}^2 \times [0,1] \longrightarrow \mathbb{R}^2, \quad
		\xsym{\vartheta}_{N+1},\dots,\xsym{\vartheta}_{N+4}\colon\mathbb{T}^2 \longrightarrow \mathbb{R}^2, 
	\]
	with physical localization
	\begin{gather*}
		\bigcup_{l=1}^{N} \operatorname{supp}(\xsym{\vartheta}_l) \subset \omegaup \times (0, 1), \quad \operatorname{supp}(\xsym{\vartheta}_{N+1}) \cup \operatorname{supp}(\xsym{\vartheta}_{N+1}) \subset \omegaup, \\ \operatorname{supp}(\xsym{\vartheta}_{N+3}) \cup \operatorname{supp}(\xsym{\vartheta}_{N+4}) \subset \Omega,
	\end{gather*}
	while $\gamma_1,\dots,\gamma_{N+4} \in \xLtwo((0, T_{\operatorname{ctrl}}); \mathbb{R})$ and $\sigma > T_{\operatorname{ctrl}}^{-1}$ are the actual controls. The expression \eqref{equation:IntroductionLocalizedControl2} is meaningful because we shall ensure that $\gamma_1(t) = \dots = \gamma_{N}(t) = 0$ for $t \leq T_{\operatorname{ctrl}}-\sigma^{-1}$ (\cf~\eqref{equation:zetagam1}).

	\begin{thrm}\label{theorem:main0_vel}
		Let $r \in \mathbb{N}_0$, $T_{\operatorname{ctrl}} > 0$, $\xvec{u}_0, \xvec{u}_1 \in \bVn{{r}}$, $\xvec{f} \in \xLtwo((0,T_{\operatorname{ctrl}});\xHn{{r+2}}(\mathbb{T}^2;\mathbb{R}^2))$, $\nu > 0$, and $\varepsilon > 0$. There exist
		\[
			\gamma_1,\dots,\gamma_{N+4} \in \xLtwo((0, T_{\operatorname{ctrl}}); \mathbb{R}), \quad \sigma \geq T_{\operatorname{ctrl}}^{-1}
		\]
		such that the unique solution $\xvec{u} \in \xCzero([0,T_{\operatorname{ctrl}}];\bVn{{r}})\cap\xLtwo((0,T_{\operatorname{ctrl}});\bVn{{r+1}})$ to the Navier--Stokes problem \eqref{equation:NavierStokesVelocity} with the external force $\xsym{\xi}$ from \eqref{equation:IntroductionLocalizedControl2}
		satisfies
		\[
			\|\xvec{u}(\cdot, T_{\operatorname{ctrl}}) - \xvec{u}_1\|_{r} < \varepsilon.
		\]
		Moreover, $\gamma_{N+1}, \dots, \gamma_{N+4}$ can be explicitly determined from $\gamma_1,\dots,\gamma_{N}, \sigma$.
	\end{thrm}

	To our knowledge \Cref{theorem:main0_vel} provides a first result towards Agrachev's problem from \cite[Section 7]{Agrachev2014}. As requested there, we only employ a finite number of actual controls, which are scalar functions merely of time, while the overall control force is physically localized. However, strictly speaking, the original version of Agrachev's problem is not fully resolved here, as our profiles $\xsym{\vartheta}_1,\dots,\xsym{\vartheta}_N$ depend on time, while \cite{Agrachev2014} asks for a control of the type \eqref{equation:truefinitedimensionalcontrol} with time-independent modes. 

	\begin{rmrk}\label{remark:energysolution}
		To simplify the presentation, weak notions of solutions to \eqref{equation:NavierStokesVelocity} are not discussed here, and we shall assume $r \geq 2$ in the proof of \Cref{theorem:main0_vel}. This is without loss of generality, since $\xvec{f} \in \xLtwo((0,T_{\operatorname{ctrl}});\xHn{{2}}(\mathbb{T}^2;\mathbb{R}^2))$ and parabolic smoothing effects ensure that the uncontrolled solution to \eqref{equation:NavierStokesVelocity} with initial state $\xvec{u}_0 \in \bVn{{0}}$ belongs at any time $t > 0$ to $\bVn{{2}}$; \eg, see \cite{Temam2001}.
	\end{rmrk}

	\subsection{Methodology} 
	Our strategy for showing \Cref{theorem:main0_vel} is newly developed and promotes viewing the building blocks $\xsym{\vartheta}_1, \dots, \xsym{\vartheta}_N$ from \eqref{equation:IntroductionLocalizedControl2} as \enquote{transported Fourier modes}. This involves explicit ad-hoc constructions and the following main ingredients:
	\begin{itemize}
		\item the return method introduced by Coron in \cite{Coron1992} for the stabilization of a mechanical system, and which has thereafter been applied to numerous nonlinear partial differential equations (\cf~\cite[Part 2, Chapter 6]{Coron2007});
		\item a linear test for approximate controllability developed in \cite{Nersesyan2021}, which involves the notion of observable families introduced in \cite{KuksinNersesyanShirikyan2020};
		\item a simplified saturation property without length condition (\cf~\Cref{subsection:saturation});
		\item a convection strategy based on rigid translations of the torus (\cf~\Cref{theorem:returnmethodflow});
		\item the localization of finite-dimensional controls via careful modifications along certain flow maps (\cf~\Cref{subsection:localizingctrl}).
	\end{itemize}
	In our approach, the controllability problem for the velocity is reduced to one for the vorticity, which then allows to utilize the underlying transport mechanisms of the considered fluid model. We linearize the vorticity equation around a special return method profile and derive a related controllability problem for a scalar transport equation. However, in order to determine the parameters $\gamma_1,\dots,\gamma_{N+4}$ in \eqref{equation:IntroductionLocalizedControl2}, we have to consider a modified transport problem where convection takes place along a vector field encoding a certain observability property. Concerning the latter, we obtain finite-dimensional controls of the type
	\begin{equation}\label{equation:fintype}
		g = \sum_{\xsym{\ell}\in\mathcal{K}}(\zeta_{\xsym{\ell}}^s(t) \sin(\xsym{\ell}\cdot \xvec{x}) + \zeta_{\xsym{\ell}}^c(t) \cos(\xsym{\ell}\cdot \xvec{x})),
	\end{equation}
	where~$\xsym{\ell} = (\ell_1,\ell_2)$ runs over a finite subset~$\mathcal{K}~\subset~\mathbb{Z}^2\setminus\{\xsym{0}\}$, providing the Fourier modes that are allowed to be triggered by the controls. The coefficients $(\zeta_{\xsym{\ell}}^s, \zeta_{\xsym{\ell}}^c)$ in \eqref{equation:fintype}, which will be used to calculate $\gamma_1,\dots,\gamma_{N+4}$, depend only on time and can be viewed as the actual controls. 
	Subsequently, the effects of the force $g$ are localized in the control region~$\omegaup$. To this end, the drift vector field encoding the observability property is interchanged with a return method trajectory shifting the whole torus~$\mathbb{T}^2$ such that all information passes through $\omegaup$. The latter step will motivate the particular constructions of $\xsym{\vartheta}_1,\dots,\xsym{\vartheta}_{N}$, which are essentially the Fourier modes from \eqref{equation:fintype} composed with a flow map translating between the observable and the return method drift. The additional profiles $\xsym{\vartheta}_{N+1}$ and $\xsym{\vartheta}_{N+2}$ render the force acting on the vorticity average-free, while $\xsym{\vartheta}_{N+3}$ and $\xsym{\vartheta}_{N+4}$ are curl-free and allow to manipulate the velocity average. Next, scaling arguments related to the return method yield  the approximate controllability in a short time of the original nonlinear vorticity problem driven by large finitely decomposable controls. Eventually, the proof of \Cref{theorem:main0_vel} is concluded by deriving formulas for the velocity controls based on the previously found vorticity controls.

	\subsection{Related literature} 
	There exists a rich literature on controllability problems for incompressible Navier--Stokes and Euler systems. A traditional point of view has been to seek the controls as abstract elements of infinite-dimensional function spaces, with significant attention being devoted to the case of bounded domains. Typically, for such a setup, one aims to obtain interior controls localized within a small subdomain or controls acting on a small part of the domain's boundary. In this context, J.-L.~Lions~\cite{LionsJL1991} raised several open problems which have so far inspired more than three decades of fruitful research, while key questions such as the global approximate controllability for the Navier--Stokes system with the no-slip boundary condition remain open until the present day. Substantial milestones have been accomplished by Coron in~\cite{Coron1996EulerEq,Coron96}, where he obtains by way of his return method the global exact and approximate controllability for planar Euler and Navier--Stokes problems respectively. The $2$D Navier--Stokes equations on manifolds without boundary are specifically covered by Coron and Fursikov in \cite{CoronFursikov1996}. Glass further developed the return method in \cite{Glass2000} to address three-dimensional configurations. A different point of view has been pursued by Lions and Zuazua in \cite{LionsZuazua1998}, where they achieved exact controllability results for Galerkin's approximations of incompressible Navier--Stokes problems by combining duality arguments with a contraction mapping principle. Several recent advances regarding the incompressible Navier--Stokes system are due to Coron~\etal~in \cite{CoronMarbachSueur2020} and Liao \etal~in \cite{LiaoSueurZhang2022}, noting that this list is far away from being comprehensive and that many past, as well as contemporary, developments are captured by the references therein.

	The question of controllability by finite-dimensional controls supported in the entire domain constitutes a subject of active research, as well. In this case, the controls are sought to be of a very specific form (\cf~\eqref{equation:truefinitedimensionalcontrol}), but are, so far, not physically localized in a given subdomain. A major breakthrough in this direction has been achieved for two-dimensional periodic domains by Agrachev and Sarychev in~\cite{AgrachevSarychev2005,AgrachevSarychev2006,AgrachevSarychev2008}, who developed a geometric control approach that has subsequently been extended and improved in various ways. For instance, drawing upon the geometric arguments due to Agrachev and Sarychev, three-dimensional Navier--Stokes problems have been treated by Shirikyan in~\cite{Shirikyan2006,Shirikyan2007} and perfect fluids by Nersisyan in~\cite{Nersisyan2011}. A new proof for results of the same type has been presented recently by Nersesyan in~\cite{Nersesyan2021}, where Coron's return method is employed in combination with a special linear test. Moreover, Phan and Rodrigues consider in~\cite{PhanRodrigues2019} a Navier--Stokes problem which is posed in a cubical domain with the Lions boundary conditions. A concise and self-contained account of the Agrachev--Sarychev method, elaborating the example of the one-dimensional Burgers equation, is provided in \cite{Shirikyan2018}.
	
	So far, the previous two paragraphs resemble rather disjoint lines of research. The notion of transported Fourier modes, which are the building blocks of our control force, is in this regard intended to connect concepts from both worlds. At least three immediate questions remain. Is it possible to build $\xsym{\vartheta}_1,\dots,\xsym{\vartheta}_N$ independent of time? Can one dispense with the profiles $\xsym{\vartheta}_{N+3}$ and $\xsym{\vartheta}_{N+4}$ that are here necessarily supported along smooth cuts? How to treat flows in the presence of physical boundaries?

	\subsection{Organization of this article} 
	In \Cref{section:transportobservabledrift}, the controllability via finite-dimensional controls is shown for transport equations with special drifts that entail a type of observability property. Next, in \Cref{section:localizedcontrols}, localized finitely decomposable controls are constructed for transport problems with special return method drift. Subsequently, in \Cref{section:Navier--Stokes}, hydrodynamic scaling arguments are setup for the construction of finitely decomposable controls for the Navier--Stokes system in vorticity form. Finally, in \Cref{section:General}, the proof of \Cref{theorem:main0_vel} is completed.

	\section{Transport with generating drift: finite-dimensional controls}\label{section:transportobservabledrift}
	For any given non-empty finite subset $\mathcal{K} \subset \mathbb{Z}^2_* \coloneq \mathbb{Z}^2 \setminus \{\xvec{0}\}$, we define a finite-dimensional space via
	\begin{gather}\label{equation:HK}
		\mathcal{H}(\mathcal{K}) \coloneq \operatorname{span} \left\{ s_{\xsym{\ell}}, c_{\xsym{\ell}} \, \left| \right. \, \xsym{\ell} \in \mathcal{K} \right\}, \\
			s_{\xsym{\ell}}(\xvec{x}) \coloneq \sin( \xsym{\ell} \cdot \xvec{x}), \quad c_{\xsym{\ell}}(\xvec{x}) \coloneq \cos( \xsym{\ell} \cdot \xvec{x}), \quad \xvec{x} \in \mathbb{T}^2, \quad \xsym{\ell} \in \mathbb{Z}_{*}^2.\label{equation:deftrigfam}
	\end{gather}
	In particular, one has $\mathcal{H}(\mathcal{K}) = \mathcal{H}(-\mathcal{K})$. Under the assumption that $\mathcal{K}$ satisfies a saturation property introduced subsequently in \Cref{subsection:saturation} (\eg, satisfied by $\mathcal{K} = \{ [1,0]^{\top}, [0,1]^{\top}\}$), we shall obtain in this section a controllability result for linear transport equations of the form
	\[
		\partial_t v + (\overline{\xvec{y}}^{\star}\cdot \xnab) v = g,
	\]
	where $g$ is a $\mathcal{H}(\mathcal{K})$-valued control and $\overline{\xvec{y}}^{\star}$ denotes a generating vector field as defined below in \Cref{subsection:observable_families}.
	
	\subsection{Saturation property}\label{subsection:saturation}
	A non-decreasing sequence of finite subsets $(\mathcal{K}_j)_{j\in\mathbb{N}_0} \subset \mathbb{Z}^2_*$ is recursively defined by means of $\mathcal{K}_0 \coloneq \mathcal{K} \cup (-\mathcal{K})$ and
	\begin{equation}\label{EE_j}
		\mathcal{K}_j \coloneq \mathcal{K}_{j-1} \cup \left\{ \xsym{\ell}_1 + \xsym{\ell}_2 \, \left| \right. \, \xsym{\ell}_1 \in \mathcal{K}_{j-1}, \xsym{\ell}_2 \in \mathcal{K}_0, \xsym{\ell}_1 \nparallel \xsym{\ell}_2 \right\}, \quad j \in \mathbb{N},
	\end{equation}
	where the constraint $\xsym{\ell}_1 \nparallel \xsym{\ell}_2$ (meaning: $\xsym{\ell}_1$ and $\xsym{\ell}_2$ are not parallel) ensures~$\mathcal{K}_j\subset\mathbb{Z}^2_*$ for each $j \in \mathbb{N}$.  Associated with $(\mathcal{K}_j)_{j\in\mathbb{N}_0}$, a sequence of subspaces $(\mathcal{H}_j(\mathcal{K}))_{j\in\mathbb{N}_0}$ is introduced via
	\[
		\mathcal{H}_j(\mathcal{K}) \coloneq \mathcal{H}(\mathcal{K}_j), \quad j \in \mathbb{N}_0.
	\]
	
	We provide a simple characterization of subsets $\mathcal{K} \subset \mathbb{Z}^2_*$ for which $\cup_{j \in \mathbb{N}_0} \mathcal{H}_j(\mathcal{K})$ is dense in $\xHn{m}$ for any $m \in \mathbb{N}_0$, where we recall that $\xHn{m}$  is the space of $\xHn{m}(\mathbb{T}^2;\mathbb{R})$ functions with zero average (\cf~\Cref{subsection:notations}).
	
	\begin{dfntn}\label{definition:saturation}
		We say that a finite subset $\mathcal{K} \subset \mathbb{Z}^2_*$ is
		\begin{enumerate}[a)]
			\item a generator if $\operatorfont{span}_{\mathbb{Z}}(\mathcal{K}) = \mathbb{Z}^2$;
			\item saturating if $\cup_{j \in \mathbb{N}_0} \mathcal{H}_j(\mathcal{K})$ is dense in $\xHn{m}$ for each $m \in \mathbb{N}_0$.
		\end{enumerate}
	\end{dfntn}
	
	The proof of the next lemma is obvious and rests on the fact that the collection $(s_{\xsym{\ell}},c_{\xsym{\ell}})_{\xsym{\ell} \in \mathbb{Z}_{*}^2}$ specified in \eqref{equation:deftrigfam} constitutes a complete orthogonal system in $\xHn{m}$ for each $m \in \mathbb{N}_0$. 
	
	\begin{lmm}\label{lemma:saturation}
		A finite subset $\mathcal{K} \subset \mathbb{Z}^2_*$ is saturating if and only if it is a generator.
	\end{lmm} 
	
	\begin{rmrk}
		In contrast to previous literature such as~\cite{AgrachevSarychev2005,AgrachevSarychev2006}, where saturating sets have been employed in the context of two-dimensional Navier--Stokes and Euler systems, no condition on the length of vectors belonging to the generator is required here.  The reason that we can use this simpler characterization of saturation is our different underlying approach, which is based on a type of linear test involving vector fields that are constructed from observable families, as introduced in the next section.
	\end{rmrk}

	\subsection{Vector fields obtained from observable families}\label{subsection:observable_families}
	We recall the notion of observable families that has been introduced in \cite{KuksinNersesyanShirikyan2020} for the study of ergodicity properties of randomly forced partial differential equations; see also \cite{Nersesyan2021} for an application to the controllability of the Navier--Stokes equations.
	\begin{dfntn}\label{definition:observable}
		Let $T > 0$ and $n \in \mathbb{N}$ be fixed. A family $(\phi_j)_{j \in \{1,\dots,n\}} \subset \xLtwo((0,T);\mathbb{R})$ is called observable if, for each subinterval $\xJ \subset (0,T)$, for any~$b \in \xCzero(\xJ;\mathbb{R})$, and for all~$(a_j)_{j \in \{1,\dots,n\}} \subset \xCone(\xJ; \mathbb{R})$, one has the implication
		\[
			b + \sum_{j=1}^n a_j \phi_j = 0 \, \mbox{ in } \xLtwo(\xJ;\mathbb{R}) \quad 	\Longrightarrow \quad \forall t \in \xJ\colon b(t) = a_1(t) = \cdots = a_n(t) = 0.
		\]
	\end{dfntn}
	
	\begin{rmrk}
		As explained in \cite[Section 3.3]{Nersesyan2021}, see also \cite{KuksinNersesyanShirikyan2020}, one can construct observable families in the sense of \Cref{definition:observable} in an explicit way and express them by means of closed formulas. For the sake of completeness, we briefly recall such a construction. Let $(\phi_j)_{j \in \{1,\dots,n\}}$ be any family of bounded measurable functions $(0,T)\longrightarrow \mathbb{R}$ such that each $\phi_j$ has well-defined left and right limits in the whole interval $(0,T)$. Next, take an arbitrary collection $(\xD_j)_{j\in\{1,\dots,n\}} \subset (0, T)$ of disjoint countable sets which are dense in $(0,T)$, and which are chosen such that $\phi_j$ is discontinuous on~$\xD_j$ while, at the same time, being continuous on $(0,T)\setminus \xD_j$. Then, one can show that~$(\phi_j)_{j \in \{1,\dots,n\}}$ constitutes an observable family. 
	\end{rmrk}
	
	\paragraph{Construction of $\overline{\xvec{y}}^{\star}$.} For $T > 0$ and a finite saturating set $\mathcal{K} \subset \mathbb{Z}_{*}^2$, we fix an observable family $(\phi_{\xsym{\ell}}^{\operatorname{s}},\phi_{\xsym{\ell}}^{\operatorname{c}})_{\xsym{\ell} \in \mathcal{K}} \subset \xLtwo((0,T); \mathbb{R})$ and choose $\phi \in \xCone([0,T];\mathbb{R})$ such that $\phi(t) = 0$ if an only if $t = T$. Then, we define the divergence-free vector field
	\begin{equation}\label{equation:DefinitionYStar}
		\overline{\xvec{y}}^{\star}(\xvec{x}, t) \coloneq \sum_{\xsym{\ell} \in \mathcal{K}} \left(\psi_{\xsym{\ell}}^{\operatorname{s}}(t) s_{\xsym{\ell}}(\xvec{x}) \xsym{\ell}^{\perp} + \psi_{\xsym{\ell}}^{\operatorname{c}}(t) c_{\xsym{\ell}}(\xvec{x}) \xsym{\ell}^{\perp} \right) \in \xWn{{1,2}}((0,T);\xCinfty(\mathbb{T}^2;\mathbb{R}^2)),
	\end{equation}
	where the coefficients $(\psi_{\xsym{\ell}}^{\operatorname{s}},\psi_{\xsym{\ell}}^{\operatorname{c}})_{\xsym{\ell} \in \mathcal{K}} \subset \xWn{{1,2}}((0,T);\mathbb{R})$ are given by
	\begin{equation*}\label{equation:definitionofpsi}
		\psi_{\xsym{\ell}}^{\operatorname{s}}(t) \coloneq \phi(t) \int_0^t \phi_{\xsym{\ell}}^{\operatorname{s}}(\sigma) \, \xdx{\sigma}, \quad \psi_{\xsym{\ell}}^{\operatorname{c}}(t) \coloneq \phi(t) \int_0^t \phi_{\xsym{\ell}}^{\operatorname{c}}(\sigma) \, \xdx{\sigma}, \quad \xsym{\ell} \in \mathcal{K},
	\end{equation*}
	and $\xsym{\ell}^{\perp} \coloneq [-\ell_2, \ell_1]^{\top}$ for any $\xsym{\ell} = [\ell_1, \ell_2]^{\top} \in \mathcal{K}$. 
	
	We show next that $\mathcal{H}(\mathcal{K})$-valued controls can generate all desired frequencies when acting on linear transport equations with drift $\overline{\xvec{y}}^{\star}$ obtained from an observable family as described above. This consequence motivates calling~$\overline{\xvec{y}}^{\star}$ a generating vector field.

	\subsection{Existence of finite-dimensional controls}
	Given $T > 0$ and a finite saturating set $\mathcal{K} \subset \mathbb{Z}_{*}^2$, let $\overline{\xvec{y}}^{\star} \in \xWn{{1,2}}((0,T);\xCinfty(\mathbb{T}^2;\mathbb{R}^2))$ be the divergence-free vector field from \eqref{equation:DefinitionYStar}. The next theorem provides $\mathcal{H}(\mathcal{K})$-valued controls for a linear transport problem with drift $\overline{\xvec{y}}^{\star}$; these controls will  act everywhere in $\mathbb{T}^2$.
	\begin{thrm}\label{lemma:FiniteDimensionalControlsTransport}
		For any $m \in \mathbb{N}$, $v_T \in \xHn{{m}}$, and $\varepsilon > 0$, there exists a control $g \in \xLtwo((0,T); \mathcal{H}(\mathcal{K}))$ such that the solution $v \in \xCzero([0,T];\xHn{{m}})\cap\xWn{{1,2}}((0,T);\xHn{{m-1}})$
		to the linear transport problem
		\begin{equation}\label{equation:TransportFC}
			\partial_t v + (\overline{\xvec{y}}^{\star}\cdot \xnab) v = g,
			\quad v(\cdot, 0) = 0
		\end{equation}
		satisfies 
		\begin{equation*}\label{equation:TransportFCCond}
			\|v(\cdot,T)-v_T\|_{m} < \varepsilon.
		\end{equation*}
	\end{thrm}
	\begin{proof}
		We employ a similar strategy as developed in \cite{Nersesyan2021}*{Section 3.3}, thereby relying on the fact that $\overline{\xvec{y}}^{\star}$ has been constructed based on an observable family. First, we denote by
		\[
			{R}(t,\tau) \colon \xHn{m} \longrightarrow \xHn{m}, \quad \widetilde{v}_0 \longmapsto \widetilde{v}(t), \quad 0 \leq \tau \leq t \leq T
		\]
		the two-parameter family of resolving operators in $\xHn{m}$ for the linear problem
		\begin{equation*}\label{equation:LinearizedStartingAtTau}
			\partial_t \widetilde{v} + (\overline{\xvec{y}}^{\star} \cdot \xnab)\widetilde{v} = 0, \quad
			\widetilde{v}(\cdot, \tau) = \widetilde{v}_0.
		\end{equation*}
		Consequentially, the resolving operator ${A} \colon \xLtwo((0,T);\xHn{m}) \longrightarrow \xHn{m}$ providing the solution to \eqref{equation:TransportFC} at time $t = T$ is by Duhamel's principle of the form
		\[
			{A}g = \int_0^{T} {R}(T, \tau) g(\tau) \, \xdx{\tau}, \quad g \in \xLtwo((0,T);\xHn{m}).
		\]
		After denoting the orthogonal projector onto $\mathcal{H}_0(\mathcal{K}) = \mathcal{H}(\mathcal{K})$ with respect to $\xHn{m}$ as
		\[
			{P}_{\mathcal{H}_0(\mathcal{K})}\colon \xHn{m} \longrightarrow \mathcal{H}_0(\mathcal{K}) \subset \xHn{m},
		\]
		we introduce the control-to-state operator
		\[
			{A}_1 \coloneq {A} {P}_{\mathcal{H}_0(\mathcal{K})} \colon \xLtwo((0,T);\xHn{m}) \longrightarrow \xHn{m}.
		\]
		It remains to show that the image of ${A}_1$ is dense in $\xHn{m}$. The latter property is equivalent to $\operatorname{ker}({A}_1^{\ast}) = \{0\}$, and the adjoint operator ${A}_1^* \colon \xHn{m} \longrightarrow \xLtwo((0,T);\xHn{m})$ can be expressed by means of
		\[
			({A}_1^* z) (\cdot) = {P}_{\mathcal{H}_0(\mathcal{K})} {R}(T,\cdot)^{*} z,
		\]
		where ${R}(T,\tau)^*$ denotes for $\tau \in [0, T]$ the $\xHn{m}$-adjoint of ${R}(T,\tau)$.
		\paragraph{Step 1. The idea for showing $\operatorname{ker}(A_1^*) = \{0\}$.}
		Let $z \in \operatorname{ker}(A_1^*)$ be arbitrarily fixed. Then, one has
		\begin{equation*}
			\begin{aligned}
				\left\langle {R}(T,\tau)g, z \right\rangle_m = \left\langle {P}_{\mathcal{H}_0(\mathcal{K})}g, {R}(T,\tau)^*z \right\rangle_m = \left\langle g, ({A}_1^*z)(\tau) \right\rangle_m = 0
			\end{aligned}
		\end{equation*}
		for almost all $\tau \in [0,T]$ and all $g \in \mathcal{H}_0(\mathcal{K})$. In fact, because $\tau \longmapsto R(T,\tau)$ is continuous,
		\begin{equation}\label{equation:orthogonalityH}
			\begin{aligned}
				\left\langle {R}(T,\tau)g, z \right\rangle_m = 0
			\end{aligned}
		\end{equation}
		for all $\tau \in [0,T]$ and $g \in \mathcal{H}_0(\mathcal{K})$.
		Consequently, taking $\tau = T$ in \eqref{equation:orthogonalityH} yields
		\begin{equation}\label{equation:orthogonalityHc1}
			\langle g, z \rangle_m = 0, \quad g \in \mathcal{H}_0(\mathcal{K}).
		\end{equation}
		Therefore, $z$ is orthogonal to $\mathcal{H}_0(\mathcal{K})$ in $\xHn{m}$. Since the inclusion~$\cup_{i \in \mathbb{N}_0}\mathcal{H}_i(\mathcal{K}) \subset \xHn{m}$ is dense by \Cref{lemma:saturation}, we can conclude $z = 0$ by showing orthogonality relations of the type \eqref{equation:orthogonalityHc1} for all $g \in \cup_{i \in \mathbb{N}_0}\mathcal{H}_i(\mathcal{K})$.
		
		Given any $T_1 \in [0,T]$, let us distinguish the element $z_1 \coloneq {R}(T,T_1)^*z$, which satisfies due to \eqref{equation:orthogonalityH} the relations
		\begin{equation}\label{equation:OrthogonalityRelationz0}
			\left\langle {R}(T_1,\tau) g, z_1 \right\rangle_m = \left\langle {R}(T,\tau)g, z \right\rangle_m = 0
		\end{equation}
		for all $\tau \in [0,T_1]$ and $g \in \mathcal{H}_0(\mathcal{K})$. Taking $\tau = T_1$ in \eqref{equation:OrthogonalityRelationz0}, we find that
		\begin{equation}\label{equation:OrthogonalityRelationz1}
			\left\langle g, z_1 \right\rangle_m = \left\langle {R}(T_1,T_1) g, z_1 \right\rangle_m = 0, \quad g \in \mathcal{H}_0(\mathcal{K}).
		\end{equation}
		If $\left\langle g, z_1 \right\rangle_m = 0$ from \eqref{equation:OrthogonalityRelationz1} would be valid for all~$g \in \mathcal{H}_i(\mathcal{K})$ with $i \in \mathbb{N}$, then, by inserting $T_1 = T$ into the definition of~$z_1$, also~\eqref{equation:orthogonalityHc1} would hold for all $g \in \mathcal{H}_i(\mathcal{K})$. This motivates the next step, where we shall verify by induction over $i \in \mathbb{N}$ that
		\begin{equation}\label{equation:OrthogonalityRelationzi}
			\left\langle g, z_1 \right\rangle_m = \left\langle {R}(T_1,T_1) g, z_1 \right\rangle_m = 0, \quad g \in \mathcal{H}_i(\mathcal{K}), \quad i \in \mathbb{N}.
		\end{equation}
		
		\paragraph{Step 2. Induction base.} We begin with establishing \eqref{equation:OrthogonalityRelationzi} for $i = 1$. To this end, we arbitrarily fix $g \in \mathcal{H}_0(\mathcal{K})$ and consider 
		\[
			q(t, \tau) \coloneq {R}(t + \tau, \tau) g, \quad Q(t, \tau) \coloneq \partial_{\tau}q(t, \tau) 
		\] 
		as $\xHn{m}$-valued functions of $t \in [0,T-\tau]$, depending on a parameter $\tau \in [0,T)$. In particular, they solve the initial value problems
		\begin{equation*}\label{equation:PdeForQ}
			\begin{cases}
				\partial_t q(t ,\tau) + (\overline{\xvec{y}}^{\star}(\cdot, t+\tau) \cdot \xnab) q(t, \tau) = 0,\\
				\partial_t Q(t ,\tau) + (\overline{\xvec{y}}^{\star}(\cdot, t+\tau) \cdot \xnab) Q(t, \tau) +  (\partial_t\overline{\xvec{y}}^{\star}(\cdot, t+\tau) \cdot \xnab) q(t, \tau) = 0,\\
				q(0, \tau) = g, \quad
				Q(0, \tau) = 0
			\end{cases}
		\end{equation*}
		for $t \in (0, T-\tau)$. After integrating from $t= 0$ to $t=T_1-\tau$ the $\xHn{m}$-inner product of the equation satisfied by~$Q$ with~$z_1$, it follows that
		\begin{equation}\label{equation:equalityobs0}
			\begin{aligned}
				0 & = \left\langle Q(T_1-\tau,\tau), z_1 \right\rangle_m + \int_{0}^{T_1 - \tau} \left\langle (\overline{\xvec{y}}^{\star}(\cdot, t+\tau) \cdot \xnab) Q(t, \tau), z_1 \right\rangle_m  \, \xdx{t} \\
				& \quad + \int_{0}^{T_1 - \tau} \left\langle (\partial_t\overline{\xvec{y}}^{\star}(\cdot, t + \tau) \cdot \xnab) q (t, \tau), z_1 \right\rangle_m \, \xdx{t}\\
				& = \left\langle Q(T_1-\tau,\tau), z_1 \right\rangle_m + \int_{0}^{T_1 - \tau} \left\langle (\overline{\xvec{y}}^{\star}(\cdot, t+\tau) \cdot \xnab) Q(t, \tau), z_1 \right\rangle_m  \, \xdx{t} \\
				& \quad + \int_{\tau}^{T_1} \left\langle (\partial_t\overline{\xvec{y}}^{\star}(\cdot, t) \cdot \xnab) [{R}(t, \tau)g], z_1 \right\rangle_m \, \xdx{t}.
			\end{aligned}
		\end{equation}
		Moreover, by using the definition of $Q$ and calculating $\partial_{\tau} ({R}( t+\tau , \tau) g)$ via the chain rule, we find that
		\begin{equation}\label{equation:FormulaForPartialTauRAtT1}
			\partial_{\tau}{R}(T_1, \tau)g = Q(T_1 - \tau, \tau) + (\overline{\xvec{y}}^{\star}(\cdot, T_1) \cdot \xnab) [{R}(T_1, \tau)g].
		\end{equation}
		Next, we take $\partial_{\tau}$ in $\langle {R}(T_1,\tau) g, z_1 \rangle_m = 0$ from \eqref{equation:OrthogonalityRelationz0} and plug \eqref{equation:FormulaForPartialTauRAtT1} into \eqref{equation:equalityobs0}, which implies
		\begin{equation*}
			\begin{aligned}
				0 & = \int_{0}^{T_1 - \tau} \left\langle (\overline{\xvec{y}}^{\star}(\cdot, t+\tau) \cdot \xnab) Q(t, \tau), z_1 \right\rangle_m  \, \xdx{t} \\
				& \quad + \int_{\tau}^{T_1} \left\langle (\partial_t\overline{\xvec{y}}^{\star}(\cdot, t) \cdot \xnab) [{R}(t, \tau)g], z_1 \right\rangle_m \, \xdx{t} - \left\langle (\overline{\xvec{y}}^{\star}(\cdot, T_1) \cdot \xnab) [{R}(T_1, \tau)g], z_1 \right\rangle_m.
			\end{aligned}
		\end{equation*}
		Differentiating the latter relation with respect to $\tau$, while invoking the definition of the vector field $\overline{\xvec{y}}^{\star}$ from \eqref{equation:DefinitionYStar}, yields
		\begin{equation}\label{equation:RepresentationViaObservableFamily}
			b(\tau) + \sum_{\xsym{\ell} \in \mathcal{K}} \left(a_{\xsym{\ell}}^{s}(\tau) \phi_{\xsym{\ell}}^{s}(\tau) + a_{\xsym{\ell}}^{c}(\tau) \phi_{\xsym{\ell}}^{c}(\tau) \right) = 0,
		\end{equation}
		where $b \in \xCzero([0,T_1];\mathbb{R})$ is the function
		\begin{equation*}
			\begin{aligned}
				b(\tau) & \coloneq \partial_{\tau} \int_{0}^{T_1 - \tau} \left\langle (\overline{\xvec{y}}^{\star}(\cdot, t+\tau) \cdot \xnab) Q(t, \tau), z_1 \right\rangle_m  \, \xdx{t} \\
				& \quad - \sum_{\xsym{\ell} \in \mathcal{K}} \left\langle \left(\phi'(\tau) \int_0^{\tau}\left(\phi_{\xsym{\ell}}^{s}(\sigma) s_{\xsym{\ell}} \xsym{\ell}^{\perp} + \phi_{\xsym{\ell}}^{c}(\sigma) c_{\xsym{\ell}}\xsym{\ell}^{\perp} \right) \xdx{\sigma}
				\cdot \xnab \right)g, z_1 \right\rangle_m \\
				& \quad + \int_{\tau}^{T_1} \left\langle (\partial_t\overline{\xvec{y}}^{\star}(\cdot, t) \cdot \xnab) [\partial_{\tau}{R}(t, \tau)g], z_1 \right\rangle_m \, \xdx{t} \\
				& \quad - \partial_{\tau} \left\langle (\overline{\xvec{y}}^{\star}(\cdot, T_1) \cdot \xnab) [{R}(T_1, \tau)g], z_1 \right\rangle_m,
			\end{aligned}
		\end{equation*}
		and the coefficients $(a_{\xsym{\ell}}^{s}, a_{\xsym{\ell}}^{c})_{\xsym{\ell} \in \mathcal{K}} \subset \xCone([0,T]; \mathbb{R})$ are given by 
		\[
			a_{\xsym{\ell}}^{s}(\tau) \coloneq -\phi(\tau) \left\langle  (s_{\xsym{\ell}} \xsym{\ell}^{\perp} \cdot \xnab)g, z_1 \right\rangle_m, \quad a_{\xsym{\ell}}^{c}(\tau) \coloneq -\phi(\tau) \left\langle (c_{\xsym{\ell}}\xsym{\ell}^{\perp} \cdot \xnab)g, z_1 \right\rangle_m.
		\]
		Since the family $(\phi_{\xsym{\ell}}^{s}, \phi_{\xsym{\ell}}^{c})_{\xsym{\ell} \in \mathcal{K}}$ is observable, the relation \eqref{equation:RepresentationViaObservableFamily} yields
		\begin{equation}\label{equation:ConvectiveTermOrthz1}
			\left\langle (s_{\xsym{\ell}} \xsym{\ell}^{\perp} \cdot \xnab) g, z_1 \right\rangle_m = \left\langle (c_{\xsym{\ell}} \xsym{\ell}^{\perp}  \cdot \xnab) g, z_1 \right\rangle_m = 0, \quad \xsym{\ell} \in \mathcal{K}.
		\end{equation}	
		Now, we take any $\xi \in \mathcal{H}_1(\mathcal{K})$ and assume without loss of generality that either $\xi = s_{\xsym{\ell}_1+\xsym{\ell}_2}$ or~$\xi = c_{\xsym{\ell}_1+\xsym{\ell}_2}$ whenever $\xsym{\ell}_1, \xsym{\ell}_2 \in \mathcal{K}$ with $\xsym{\ell}_1 \nparallel \xsym{\ell}_2$. In view of the trigonometric identities
		\begin{equation}\label{equation:ti}
			s_{\xsym{\ell}_1+\xsym{\ell}_2} = s_{\xsym{\ell}_1} c_{\xsym{\ell}_2}+c_{\xsym{\ell}_1}s_{\xsym{\ell}_2}, \quad
			c_{\xsym{\ell}_1+\xsym{\ell}_2} = c_{\xsym{\ell}_1} c_{\xsym{\ell}_2}-s_{\xsym{\ell}_1}s_{\xsym{\ell}_2},
		\end{equation}
		the function $\xi$ admits one of the two representations
		\[
			\xi = s_{\xsym{\ell}_1}c_{\xsym{\ell}_2} + c_{\xsym{\ell}_1}s_{\xsym{\ell}_2}, \quad \xi = c_{\xsym{\ell}_1}c_{\xsym{\ell}_2} - s_{\xsym{\ell}_1}s_{\xsym{\ell}_2}.
		\]
		Then, by choosing $g = c_{\xsym{\ell}_1}$ or $g = s_{\xsym{\ell}_1}$ in \eqref{equation:ConvectiveTermOrthz1}, we obtain the relation $\langle \xi, z \rangle_m = 0$, which implies \eqref{equation:OrthogonalityRelationzi} for $i = 1$ due to the arbitrariness of $\xi \in \mathcal{H}_1(\mathcal{K})$. 
		
		\paragraph{Step 3. Induction step.} With the intention of closing the induction argument, we assume now for arbitrarily fixed $i \in \mathbb{N}$ that
		\begin{equation}\label{equation:orthogonalityHc2}
			\langle g, z \rangle_m = 0, \quad g \in \mathcal{H}_i(\mathcal{K}).
		\end{equation}
		By analysis similar to the previous step, the statement in \eqref{equation:orthogonalityHc2} leads to
		\begin{equation}\label{equation:ConvectiveTermOrthz2b}
			\left\langle (s_{\xsym{\ell}} \xsym{\ell}^{\perp}  \cdot \xnab) g, z_1 \right\rangle_m = \left\langle (c_{\xsym{\ell}} \xsym{\ell}^{\perp}  \cdot \xnab) g, z_1 \right\rangle_m = 0, \quad \xsym{\ell} \in \mathcal{K}, \quad g \in \mathcal{H}_i(\mathcal{K}).
		\end{equation}	
		It remains to show that $z_1$ is orthogonal to $\mathcal{H}_{i+1}(\mathcal{K})$.  To~this end, let~$(\mathcal{K}_j)_{j \in \mathbb{N}_0}$ be defined as in \eqref{EE_j}. Using \eqref{equation:ConvectiveTermOrthz2b} with~$\xsym{\ell}=\xsym{\ell}_2\in \mathcal{K}$ and $g=s_{\xsym{\ell}_1}$ or $g=c_{\xsym{\ell}_1}$ for $\xsym{\ell}_1\in \mathcal{K}_i$ such that 
		$\xsym{\ell}_1 \nparallel \xsym{\ell}_2$, and by employing trigonometric identities of the type \eqref{equation:ti}, we observe that $z$ is orthogonal to $s_{\xsym{\ell}}$ and $c_{\xsym{\ell}}$ with any $\xsym{\ell}\in \mathcal{K}_{i+1}$. This implies that~$\langle g, z \rangle_m = 0$ for all $g \in \mathcal{H}_{i+1}(\mathcal{K})$
		and completes the proof.
	\end{proof}
	
	\begin{rmrk}\label{remark:rinv}
		Given $\varepsilon > 0$ and a bounded set $\xB$ in $\xHn{{m}}$, there exists a bounded linear operator $\mathcal{C}_{\varepsilon}$ assigning to any $v_T$ from $\xB$ a control $g \in \xLtwo((0,T);\mathcal{H}(\mathcal{K}))$ such that the corresponding solution $v$ to \eqref{equation:EulerLinearizedNonlocalized} satisfies~$\|v(\cdot, T)-v_T\|_{m-1} < \varepsilon$. To see this, we introduce the resolving operator 
		\[
			\mathcal{A}\colon \xHn{{m}} \times \xLtwo((0,T),\mathcal{H}(\mathcal{K})) \longrightarrow \xCzero([0,T];\xHn{{m}})\cap\xWn{{1,2}}((0,T);\xHn{{m-1}})
		\]
		which associates with $v_0 \in \xHn{{m}}$ and $g \in \xLtwo((0,T);\mathcal{H}(\mathcal{K}))$ the solution $v$ to the problem
		\begin{equation*}\label{equation:EulerLinearizedNonlocalizedB}
			\partial_t v + (\overline{\xvec{y}}^{\star} \cdot \xnab) v = g,
			\quad v(\cdot, 0) = v_0,
		\end{equation*}
		and further denote by $\mathcal{A}_T$ the restriction of $\mathcal{A}$ to the time $t = T$. As a result of \Cref{lemma:FiniteDimensionalControlsTransport}, the range of the linear operator $g \mapsto\mathcal{A}_T(0, g)$ is dense in $\xHn{{m}}$. Therefore, by \cite[Proposition 2.6]{KuksinNersesyanShirikyan2020} (see also the proof of Theorem 2.3 in \cite{Nersesyan2021}), there exists a bounded linear approximate right inverse $\mathcal{C}_{\varepsilon}$ for $\mathcal{A}_T(0, \cdot)$, which is as desired.
	\end{rmrk}

	\section{Transport with return method drift: localized degenerate controls}\label{section:localizedcontrols}
	The objective of this section is to fix a special vector field $\overline{\xvec{y}}$ so that all of its integral curves cross the set $\omegaup$ selected in \Cref{section:introduction}, and then to obtain a controllability result for the linear transport equation
	\begin{equation}\label{equation:previewfdc}
		\partial_t v + (\overline{\xvec{y}} \cdot \xnab) v = \mathbb{I}_{\omegaup}\widetilde{\eta},
	\end{equation}
	where $\widetilde{\eta}$ shall be a finitely decomposable control that is physically localized in $\omegaup$. 
	
	\subsection{Open covering by overlapping squares}\label{subsection:opencoveringoverlapping}
	Let us recall that $\omegaup$ is the nonempty subdomain of~$\Omega$ fixed in \Cref{subsection:controllabilityproblem}. Moreover, we take $\xvec{p}_{\omegaup} \in \mathbb{T}^2$ and fix a possibly large square number $K = K(\omegaup) \in \mathbb{N}$ such that
	\begin{equation*}\label{equation:lengthcond}
		\xvec{p}_{\omegaup} + [0, l_K]^2 \subset \omegaup, \quad l_K \coloneq \frac{2\pi}{\sqrt{K}-1}.
	\end{equation*}
	In particular, we denote by $\mathcal{O} \coloneq \xvec{p}_{\omegaup} + (0,l_K)^2$ a reference square. Then, as indicated in~\Cref{Figure:Covering}, we cover~$\mathbb{T}^2$ with overlapping squares $(\mathcal{O}_i)_{i\in\{1,\dots,K\}}$, each being a rigid translation of $(0, l_K)^2$. For the sake of explicitness, we take the bottom left corners of $\mathcal{O}_1, \dots, \mathcal{O}_K$ as $\xvec{x}_k = [x_{k,1}, x_{k,2}]^{\top}$, $k \in \{1,\dots,K\}$, with the coordinates
	\begin{equation}\label{equation:coords}
		x_{i + \sqrt{K}(l - 1),1} = \frac{2\pi(i-1)}{\sqrt{K}}, \quad x_{i + \sqrt{K}(l - 1),2} = \frac{2\pi(l-1)}{\sqrt{K}}, \quad i,l = 1, \dots, \sqrt{K}.
	\end{equation}

	\begin{figure}[ht!]
		\centering
		\resizebox{0.62\textwidth}{!}{
			\begin{tikzpicture}
				\clip(-0.5,-0.1) rectangle (5.4,5);
				
				\draw[line width=0.4mm, dashed, color=CornflowerBlue!120] plot[smooth cycle] (1.3,1.3) rectangle (4.4,4.4);
				\draw[line width=0.1mm, color=Black] plot[smooth cycle] (0.25,0) -- (0.25,1);
				\draw[line width=0.1mm, color=Black] plot[smooth cycle] (0,0) rectangle (1,1);
				\draw[line width=0.1mm, color=Black] plot[smooth cycle] (0.75,0) rectangle (1.75,1);
				\draw[line width=0.1mm, color=Black] plot[smooth cycle] (1.5,0) rectangle (2.5,1);
				\draw[line width=0.1mm, color=Black] plot[smooth cycle] (2.25,0) rectangle (3.25,1);
				\draw[line width=0.1mm, color=Black] plot[smooth cycle] (3,0) rectangle (4,1);
				\draw[line width=0.1mm, color=Black] plot[smooth cycle] (4,0) -- (4,1);
				\draw[line width=0.1mm, color=Black] plot[smooth cycle] (3.75,0) rectangle (4.75,1);
				\draw[line width=0.1mm, color=Black] plot[smooth cycle] (0.25,0.75) -- (0.25,1.75);
				\draw[line width=0.1mm, color=Black] plot[smooth cycle] (0,0.75) rectangle (1,1.75);
				\draw[line width=0.1mm, color=Black] plot[smooth cycle] (0.75,0.75) rectangle (1.75,1.75);
				\draw[line width=0.1mm, color=Black] plot[smooth cycle] (0,1.5) rectangle (1,2.5);
				\draw[line width=0.1mm, color=Black] plot[smooth cycle] (0,2.25) rectangle (1,3.25);
				\draw[line width=0.1mm, color=Black] plot[smooth cycle] (0,3) rectangle (1,4);
				\draw[line width=0.1mm, color=Black] plot[smooth cycle] (0,3.75) rectangle (1,4.75);
				\draw[line width=0.1mm, color=Black] plot[smooth cycle] (0,0.25) -- (1,0.25);
				\draw[line width=0.4mm, color=FireBrick] plot[smooth cycle] (2.35,2.35) rectangle (3.35,3.35);
				\draw[line width=0.2mm, fill=FireBrick, color=black] plot[smooth cycle]  (2.35,2.35) circle (0.06);
				\draw[line width=0.1mm, color=Black, opacity=0, postaction={pattern=crosshatch,opacity=0.2}] plot[smooth cycle] (4.5,0) rectangle (4.75,1);
				\draw[line width=0.1mm, color=Black, opacity=0, postaction={pattern=crosshatch,opacity=0.2}] plot[smooth cycle] (0,0) rectangle (0.25,1);
				\draw[line width=0.1mm, color=Black, opacity=0, postaction={pattern=dots,opacity=0.8}] plot[smooth cycle] (0,0) rectangle (1,0.25);
				\draw[line width=0.1mm, color=Black, opacity=0, postaction={pattern=dots,opacity=0.8}] plot[smooth cycle] (0,4.5) rectangle (1,4.75);
				
				\coordinate[label=below:{\scriptsize$\xvec{p}_{\omegaup}$}] (A) at (2.35,2.35);
				\coordinate[label=below:\color{FireBrick}\scriptsize$\mathcal{O}$] (B) at (2.85,3.07);
			\end{tikzpicture}
		}
		\caption{An illustration of the chosen covering of $\mathbb{T}^2$, here by $K = 36$ overlapping squares. Overlaps due to periodicity are filled with a corresponding pattern, and only a few squares are depicted. The (red) reference square $\mathcal{O}$ is located in the control region (indicated by blue dashes) and $\xvec{p}_{\omegaup}$ marks its bottom left corner.}
		\label{Figure:Covering}
	\end{figure}
	
	\subsection{Partition of unity}\label{subsection:partitionOfunity}
	We introduce a partition of unity with respect to $(\mathcal{O}_l)_{l\in\{1,\dots,K\}}$ arising from rigid translations of a single cutoff function. To begin with, let $\widetilde{\mu} \in \xCinfty(\mathbb{T};[0,1])$ have the attributes (\cf~\Cref{example:pfu})
	\begin{equation}\label{equation:cfop1}
		\operatorname{supp}(\widetilde{\mu}) \subset (0,l_K), \quad \forall x \in \mathbb{T}\colon \sum_{l=1}^{\sqrt{K}} \widetilde{\mu}\left(x + \frac{2\pi(l-1)}{\sqrt{K}}\right) = 1.
	\end{equation}
	Thereafter, we define the cutoff functions $\mu,\chi \in \xCinfty(\mathbb{T}^2;[0,1])$ by virtue of
	\begin{equation}\label{equation:Definition_chi}
		\mu(\xvec{x}) \coloneq \widetilde{\mu}(x_1) \widetilde{\mu}(x_2), \quad \chi(\xvec{x}) \coloneq \mu(\xvec{x}-\xvec{p}_{\omegaup}), \quad \xvec{x} = [x_1,x_2]^{\top}\in\mathbb{T}^2.
	\end{equation}
	As a result, a partition of unity with respect to the open covering $(\mathcal{O}_l)_{l \in \{1,\dots,K\}}$ is given by the family of translations
	\begin{equation}\label{equation:TranslationsCutoff}
		\left(\mu_l \coloneq \mu\left(\cdot - \xvec{x}_l\right)\right)_{l=1,\dots,K} \subset \xCinfty(\mathbb{T}^2;[0,1]),
	\end{equation}
	where $\xvec{x}_1, \dots, \xvec{x}_K$ are given via \eqref{equation:coords}. In particular, one has $\sum_{l=1}^{K} \mu_l = 1$.

	\begin{xmpl}\label{example:pfu}
		For $K \geq 9$, take any $\widehat{\mu} \in \xCinfty_0((0, l_K - 1/2K);[0,1])$ obeying $\widehat{\mu}(s) = 1$ if and only if $s \in [2\pi/(K-\sqrt{K}), l_K - 1/K]$. Subsequently, a cutoff $\widetilde{\mu} \in \xCinfty(\mathbb{T};[0,1])$ obeying \eqref{equation:cfop1} can be defined as
		\[
			\widetilde{\mu}(s) = \mathbb{I}_{\left[0,\frac{2\pi}{K-\sqrt{K}}\right]}(s) \widehat{\mu}(s) + \mathbb{I}_{\left(\frac{2\pi}{K-\sqrt{K}},\frac{2\pi}{\sqrt{K}}\right)}(s) + \mathbb{I}_{\left[\frac{2\pi}{\sqrt{K}},l_K\right]}(s) \left(1 - \widehat{\mu}\left(s - \frac{2\pi}{\sqrt{K}} \right)\right).
		\]
	\end{xmpl}

	\subsection{Convection strategy on the torus}\label{subsection:flushing}
	Inspired by classical applications of the return method to fluid problems (\cf~\cite[Part~2, Chapter 6.2]{Coron2007}), we shall introduce a vector field $\overline{\xvec{y}}$ along which information originating anywhere on the torus will eventually pass through the control region. To start with, the reference time interval $[0,1]$ is subdivided by the points
	\[	
		0 < T_a = t^0_c < t^1_a < t^1_b < t^1_c < t^2_a < t^2_b < t^2_c < \dots < t^K_a < 	t^K_b < t^K_c = T_b < 1,
	\]
	which are, for simplicity, chosen to be of equal distance denoted by $T^{\star} > 0$, that is
	\[
		t^l_a - t^{l-1}_c = t^l_c-t^l_b = t^l_b-t^l_a = T_a = 1 - T_b = T^{\star}, \quad l \in \{1,\dots,K\}.
	\]
	
	Now, the smooth vector field $\overline{\xvec{y}} \in \xCinfty_0((0,1); \mathbb{T}^2)$ is constructed such that~$\overline{\xvec{y}}$ together with its flow~$\xsym{\mathcal{Y}}$, obtained by solving the system of ordinary differential equations
	\begin{equation}\label{equation:FlowY}
		\xdrv{}{t} \xsym{\mathcal{Y}}(\xvec{x},s,t) = \overline{\xvec{y}}(t), \quad
		\xsym{\mathcal{Y}}(\xvec{x},s,s) = \xvec{x}, \quad s,t \in [0,1], 
	\end{equation}
	possess the following properties:
	\begin{enumerate}[P1)]
		\item\label{item:P1} $\overline{\xvec{y}}(t) = \xvec{0}$ for all $t \in [0, T_a] \cup [T_b, 1]$;
		\item\label{item:P2} $\xsym{\mathcal{Y}}(\xvec{x},0,1) = \xvec{x}$ for all $\xvec{x} \in \mathbb{T}^2$;
		\item\label{item:P3} each square $\mathcal{O}_l$ is transported by $\xsym{\mathcal{Y}}$ into $\mathcal{O} \subset \omegaup$ and pauses there during $[t^l_a, t^l_b]$, that is
		\[
			\forall l \in \{1,\dots,K\},\,  \forall t \in [t^l_a, t^l_b] \colon \,\xsym{\mathcal{Y}}(\mathcal{O}_l, 0, t) = \mathcal{O}.
		\]
	\end{enumerate}

	\begin{figure}[ht!]
		\centering
		\resizebox{1\textwidth}{!}{
			\begin{subfigure}[b]{0.5\textwidth}
				\centering
				\resizebox{0.99\textwidth}{!}{
					\begin{tikzpicture}
						\clip(-0.2,-0.2) rectangle (9.2,9.2);
						
						\draw[line width=0.8mm, dashed, color=CornflowerBlue!120] plot[smooth cycle] (3,2.9) rectangle ++(3.4,3.3);
						\coordinate[label=left:\Large\color{CornflowerBlue}${\omegaup}$] (A) at (5.1,4.5);

						\draw[line width=0.3mm, color=Black] plot[smooth cycle] (1.2,1.2) rectangle (2.4,2.4);
						\coordinate[label=left:\Large${\mathcal{O}_i}$] (D1) at (2.3,1.8);
						\draw[->,dashed,line width=1mm] (2.4,1.8) -- (4.1,1.8);
						\draw[line width=0.3mm, color=Black, opacity=0.4] plot[smooth cycle] 	(4.1,1.2) rectangle (5.3,2.4);
						\draw[line width=0.3mm, color=Black] plot[smooth cycle] (6.8,7.3) rectangle (8,8.5);
						\coordinate[label=left:\Large${\mathcal{O}_j}$] (D2) at (8,7.8);
						\draw[line width=0.3mm, color=Black, opacity=0.4] plot[smooth cycle] 	(0.7,7.3) rectangle (1.9,8.5);
						\draw[dashed,line width=1mm] (8,7.9) -- (9,7.9);
						\draw[->, dashed,line width=1mm] (0,7.9) -- (0.7,7.9);
						
						
						\draw[line width=0.2mm, fill=FireBrick, color=black] plot[smooth cycle]  (4.1,3.9) circle (0.1);
						\coordinate[label=below:{\Large$\xvec{p}_{\omegaup}$}] (AAA) at (3.9,3.9);
						
						\draw[line width=0.2mm, color=black, dashed, dash pattern=on 8pt off 8pt]	plot[smooth cycle] (0,0) rectangle ++(9,9);
						
					\end{tikzpicture}
				}
				\caption{Shifting $\mathbb{T}^2$ horizontally such that the bottom left corner of~$\mathcal{O}_i$ moves to the $x_1$ coordinate of $\xvec{p}_{\omegaup}$.}
				\label{Figure:ControlInAStripIntersection1}
			\end{subfigure}
			\hfill \quad \hfill
			\begin{subfigure}[b]{0.5\textwidth}
				\centering
				\resizebox{0.99\textwidth}{!}{
					\begin{tikzpicture}
						\clip(-0.2,-0.2) rectangle (9.2,9.2);
						
						\draw[line width=0.8mm, dashed, color=CornflowerBlue!120] plot[smooth cycle] (3,2.9) rectangle ++(3.4,3.3);
						
						\draw[line width=0.3mm, color=Black] plot[smooth cycle] (4.1,1.2) rectangle (5.3,2.4);
						\draw[->,dashed,line width=1mm] (4.7,2.4) -- (4.7,3.9);
						\draw[line width=0.3mm, color=FireBrick, opacity=0.8] plot[smooth cycle] 	(4.1,3.9) rectangle (5.3,5.1);
						\draw[line width=0.3mm, color=Black] plot[smooth cycle] (0.7,7.3) rectangle (1.9,8.5);
						\draw[dashed,line width=1mm] (1.3,8.5) -- (1.3,9);
						\draw[->, dashed,line width=1mm] (1.3,0) -- (1.3,1);
						\draw[line width=0.3mm, color=Black, opacity=0.4] plot[smooth cycle] 	(0.7,1) rectangle (1.9,2.2);
						\coordinate[label=left:\Large${\color{FireBrick}\mathcal{O}}$] (D2) at (5.1,4.48);
						
						\draw[line width=0.2mm, fill=FireBrick, color=black] plot[smooth cycle]  (4.1,3.9) circle (0.1);
						\coordinate[label=below:{\Large$\xvec{p}_{\omegaup}$}] (AAA) at (3.9,3.9);
						
						\draw[line width=0.2mm, color=black, dashed, dash pattern=on 8pt off 8pt]	plot[smooth cycle] (0,0) rectangle ++(9,9);
						
					\end{tikzpicture}
				}
				\caption{Shifting $\mathbb{T}^2$ vertically so that the $x_2$ coordinate of the bottom left corner of $\mathcal{O}_i$ is translated to the $x_2$ coordinate of $\xvec{p}_{\omegaup}$.}
				\label{Figure:ControlInAStripIntersection2}
			\end{subfigure}
		}
		\caption{The vector field $\overline{\xvec{y}}$, which is a function only of time, is constructed by a sequence of horizontal and vertical shifts of the whole torus $\mathbb{T}^2$. }
		\label{Figure:ConstructionReturnmethodFlow}
	\end{figure}
	
	\begin{thrm}\label{theorem:returnmethodflow}
		There exists $\overline{\xvec{y}} \in \xCinfty_0((0,1); \mathbb{T}^2)$ satisfying the properties {\rm P}\Rref{item:P1}--{\rm P}\Rref{item:P3}.
	\end{thrm}
	\begin{proof}
		The proof is based on translations of~$\mathbb{T}^2$, as illustrated in \Cref{Figure:ConstructionReturnmethodFlow}. One could also use other translations such as shifts in the direction of $\xvec{p}_{\omegaup} - \xvec{x}_l$ for $l \in \{1,\dots,K\}$. To begin with, for each $i \in \{1,\dots,K\}$ we fix $h_i, v_i \in \xCinfty_0((0, T^{\star}/2); \mathbb{R})$ with
		\[
			\left(x_{i,1} + \int_0^{T^{\star}/2} h_i(s) \, \xdx{s}\right)\xvec{e}_1 + \left(x_{i,2} + \int_0^{T^{\star}/2} v_i(s) \, \xdx{s}\right)\xvec{e}_2 = \xvec{p}_{\omegaup},
		\]
		where $\xvec{e}_1 \coloneq [1,0]^{\top}$ and $\xvec{e}_2 \coloneq [0,1]^{\top}$ are the standard basis vectors of $\mathbb{R}^2$. Then, we introduce $\widehat{\xsym{\Theta}}_i \in \xCinfty_0((0,3T^{\star}); \mathbb{R}^2)$ via
		\begin{equation*}
			\widehat{\xsym{\Theta}}_i(t) \coloneq \begin{cases}
			\widetilde{\xsym{\Theta}}_i(t) &  \mbox{ if } t \in [0, T^{\star}],\\
			\xsym{0} &  \mbox{ if } t \in [T^{\star},2T^{\star}],\\
			-\widetilde{\xsym{\Theta}}_i(t-2T^{\star}) & \mbox{ if } t \in (2T^{\star}, 3T^{\star}],
			\end{cases}
		\end{equation*}
		where $\widetilde{\xsym{\Theta}}_i \in \xCinfty_0((0, T^{\star}); \mathbb{R}^2)$ is given by
		\[
			\widetilde{\xsym{\Theta}}_i(t) \coloneq \begin{cases}
				h_i(t) \xvec{e}_1 &  \mbox{ if } t \in [0, T^{\star}/2],\\
				v_i(t-T^{\star}/2) \xvec{e}_2 & \mbox{ if } t \in (T^{\star}/2, T^{\star}].
			\end{cases} 
		\]
		Finally, the spatially constant vector field $\overline{\xvec{y}} \in \xCinfty_0((0,1); \mathbb{T}^2)$ defined as
		\[
			\overline{\xvec{y}}(t) \coloneq \begin{cases}
				\xsym{0} & \mbox{ if } t \in [0, T_a),\\
				\widehat{\xsym{\Theta}}_i(t - (3i-2)T^{\star}), &  \mbox{ if } t \in [(3i-2)T^{\star}, 	(3i+1)T^{\star}],\\
				\xsym{0} & \mbox{ if } t \in (T_b, 1)
			\end{cases}
		\]
		satisfies the properties {\rm P}\Rref{item:P1}--{\rm P}\Rref{item:P3}.
	\end{proof}
	From now on, let $\overline{\xvec{y}} \in \xCinfty_0((0, 1);\mathbb{R}^2)$ be fixed via \Cref{theorem:returnmethodflow} and denote by~$\xsym{\mathcal{Y}}$ its flow, which obeys \eqref{equation:FlowY}. This choice only depends on the geometry of~$\omegaup$; in particular,~$\overline{\xvec{y}}$ is independent of the data imposed in \Cref{theorem:main0_vel}. Next, a divergence-free vector field $\overline{\xvec{u}} \in \xWn{{1,2}}((0,1);\xCinfty(\mathbb{T}^2;\mathbb{R}^2))$ is introduced by
	\begin{equation}\label{equation:Definitionybarstarubar}
		\begin{aligned}
			\overline{\xvec{u}}(\cdot,t) \coloneq \overline{\xvec{y}}(t) + \mathbb{I}_{[T_b, 1]}(t)\overline{\xvec{y}}^{\star}(\cdot,t - T_b),
		\end{aligned}
	\end{equation}
	where $\overline{\xvec{y}}^{\star}$ is the profile defined in \eqref{equation:DefinitionYStar} for the choice $T = T^{\star}$ and a fixed finite saturating set $\mathcal{K} \subset \mathbb{Z}_{*}^2$. The flow $\xsym{\mathcal{U}}$ of $\overline{\xvec{u}}$ is determined via
	\begin{equation}\label{equation:flowU}
		\begin{cases}
			\xdrv{}{t} \xsym{\mathcal{U}}(\xvec{x},s,t) = \overline{\xvec{u}}(\xsym{\mathcal{U}}(\xvec{x},s,t), t),\\
			\xsym{\mathcal{U}}(\xvec{x},s,s) = \xvec{x}.
		\end{cases}
	\end{equation}

	\subsection{Localizing the controls}\label{subsection:localizingctrl}
	We state now a controllability result for \eqref{equation:previewfdc} by means of a physically localized finitely decomposable control $\widetilde{\eta} = \widetilde{\eta}_{(\zeta_{\xsym{\ell}}^s,\zeta_{\xsym{\ell}}^c)_{\xsym{\ell}\in\mathcal{K}}}$ of the form
	\begin{equation}\label{equation:ExplicitControl2}
		\begin{aligned}
			\widetilde{\eta}(\xvec{x},t) = \widehat{\eta}(\xvec{x},t) - \frac{(\overline{\xvec{y}}(t) \cdot \xnab){\chi}(\xvec{x})\int_0^t \int_{\mathbb{T}^2} \widehat{\eta}(\xvec{z}, s) \, \xdx{\xvec{z}} \xdx{s}}{\int_{\mathbb{T}^2} \chi(\xvec{z}) \, \xdx{\xvec{z}}} - \frac{{\chi}(\xvec{x})\int_{\mathbb{T}^2} \widehat{\eta}(\xvec{z}, t) \, \xdx{\xvec{z}}}{\int_{\mathbb{T}^2} \chi(\xvec{z}) \, \xdx{\xvec{z}}},
		\end{aligned}
	\end{equation}
	where $\widehat{\eta} = 	\widehat{\eta}_{(\zeta_{\xsym{\ell}}^s,\zeta_{\xsym{\ell}}^c)_{\xsym{\ell}\in\mathcal{K}}}$ is for control parameters $(\zeta_{\xsym{\ell}}^s,\zeta_{\xsym{\ell}}^c)_{\xsym{\ell}\in\mathcal{K}} \subset \xLtwo((0,1);\mathbb{R})$ given by
	\begin{equation}\label{equation:ExplicitControl3}
		\widehat{\eta}(\xvec{x},t) = \chi(\xvec{x})\sum_{j=1}^K 
		\mathbb{I}_{[t^j_a,t^j_b]}(t)  \sum_{\xsym{\ell} \in \mathcal{K}} \left[\zeta_{\xsym{\ell}}^s(\tau_j(t)) s_{\xsym{\ell}} \left(\xsym{\Xi}(\xvec{x},t) \right) + \zeta_{\xsym{\ell}}^c(\tau_j(t)) c_{\xsym{\ell}}\left( \xsym{\Xi}(\xvec{x},t)\right)\right],
	\end{equation}
	with $s_{\xsym{\ell}}$ and $c_{\xsym{\ell}}$ being the trigonometric functions from \eqref{equation:deftrigfam}, $\chi$ the cutoff obeying $\operatorname{supp}(\chi) \subset \omegaup$ given in \eqref{equation:Definition_chi}, $\mathcal{K} \subset \mathbb{Z}_{*}^2$ the finite saturating set from \Cref{subsection:flushing}, and
	\begin{gather}\label{equation:DefinitionPhiPsi}
			\xsym{\Xi}(\xvec{x},t) = [\Phi(\xvec{x},t),\Psi(\xvec{x},t)]^{\top}  \coloneq \xsym{\mathcal{U}}\left(\xsym{\mathcal{Y}}(\xvec{x},t,0), 1, \overline{\tau}(t)\right),\\
			\overline{\tau}(t) \coloneq \sum_{l=1}^{K} \tau_l(t), \quad \tau_l(t) \coloneq \mathbb{I}_{[t^l_a,t^l_b]}(t)\left(T_b + t-t^l_a \right)\label{equation:Definition_taul}
	\end{gather}
	for the flows $\xsym{\mathcal{Y}}$ and $\xsym{\mathcal{U}}$ from \Cref{subsection:flushing}.
	
	Let us briefly comment on these definitions, emphasizing that our particular choice of $\widetilde{\eta}_{(\zeta_{\xsym{\ell}}^s,\zeta_{\xsym{\ell}}^c)_{\xsym{\ell}}}$ will be motivated during the proof of \Cref{theorem:localized} below. 
	1) In the proof of \Cref{theorem:localized}, we shall at first use $\widehat{\eta}$ as control force. However, as $\widehat{\eta}$ is not average-free, and thus cannot enter the vorticity formulation of the Navier--Stokes system, we will eventually explain how $\widetilde{\eta}$ arises as a suitable average-free control from $\widehat{\eta}$.
	2) The composed flow $\xsym{\Xi}$ in \eqref{equation:DefinitionPhiPsi} shall play a key role in localizing finite-dimensional controls from \Cref{section:transportobservabledrift}. Indeed, the presence of $\xsym{\Xi}$ as a change of variables will allow us to translate between the convection mechanisms induced by the generating drift $\overline{\xvec{y}}^{\star}$ in \eqref{equation:TransportFC} and the return method drift $\overline{\xvec{y}}$ in \eqref{equation:previewfdc}. This change of drift profiles is important for our approach: $\overline{\xvec{y}}^{\star}$ is suitable for obtaining finite dimensional controls, but many of its integral curves might not cross the control region; on the other hand, one could not get finite-dimensional controls for transport equations with drift $\overline{\xvec{y}}$, but it channels all information in a well-ordered manner through $\omegaup$.

	\begin{thrm}\label{theorem:localized}
		Given $m \in \mathbb{N}$, $\varepsilon > 0$, and $v_1 \in \xHn{{m}}$, there exist $(\zeta_{\xsym{\ell}}^s,\zeta_{\xsym{\ell}}^c)_{\xsym{\ell}\in\mathcal{K}} \subset \xLtwo((0,1);\mathbb{R})$ such that the solution $v \in \xCzero([0,1];\xHn{{m}})\cap\xWn{{1,2}}((0,1);\xHn{{m-1}})$ to the transport equation
		\begin{equation}\label{equation:EulerLinearizedlocalized}
			\partial_t v + (\overline{\xvec{y}} \cdot \xnab) v = \widetilde{\eta}_{(\zeta_{\xsym{\ell}}^s,\zeta_{\xsym{\ell}}^c)_{\xsym{\ell}\in\mathcal{K}}}, \quad 
			v(\cdot, 0) = 0
		\end{equation}
		obeys
		\begin{equation}\label{equation:LinApproxEst}
			\|v(\cdot,1)-v_1\|_{m} < \varepsilon.
		\end{equation}
	\end{thrm}

	\begin{proof}
		A rough outline of the proof is as follows. First, we shall start with a controlled solution $V$ to the transport problem
		\[
			\partial_t V + (\overline{\xvec{u}} \cdot \xnab) V = \mathbb{I}_{[T_b,1]}(t)g, \quad 
			V(\cdot, 0) = 0,
		\]
		where $g$ is a finite-dimensional control from \Cref{section:transportobservabledrift} and $\overline{\xvec{u}}$ is the vector field defined in \eqref{equation:Definitionybarstarubar}. Then, by employing the method of characteristics, we have
		\begin{equation}\label{equation:exp}
			V(\xvec{x},1) = \int_0^1 (\mathbb{I}_{[T_b,1]}g)(\xsym{\mathcal{U}}(\xvec{x},1,s),s) \, \xdx{s}.
		\end{equation}
		The next step will be to explain how~$g$ can be replaced by a localized force $\widehat{\eta}$, and~$\xsym{\mathcal{U}}$ be replaced by $\xvec{\mathcal{Y}}$, without impacting the value of the integral in \eqref{equation:exp}. That is, we shall obtain a solution $\widehat{v}$ to
		\begin{equation*}
			\partial_t \widehat{v} + (\overline{\xvec{y}} \cdot \xnab) \widehat{v} = \mathbb{I}_{\omegaup}\widehat{\eta}, \quad 
			{v}(\cdot, 0) = 0
		\end{equation*}
		such that 
		\[
		\widehat{v}(\cdot,1) = \int_0^1 (\mathbb{I}_{\omegaup}\widehat{\eta})(\xsym{\mathcal{Y}}(\cdot,1,s),s) \, \xdx{s} = \int_0^1 (\mathbb{I}_{[T_b,1]}g)(\xsym{\mathcal{U}}(\cdot,1,s),s) \, \xdx{s} = \widetilde{v}(\cdot,1).
		\]
		Since $V(\cdot, 1)$ will by the choice of $g$ be made sufficiently close to the target state, the same shall be true for $\widehat{v}(\cdot, 1)$. Finally, the same target state will also be achieved using an average-free version~$\widetilde{\eta}$ of~$\widehat{\eta}$.
		
		\paragraph{Step 1. Determining the control parameters.}
		Let us recall from \eqref{equation:Definitionybarstarubar} that the vector field $\overline{\xvec{u}}(\cdot, t)$ equals to $\overline{\xvec{y}}(t)$ for $t \in [0, T_b)$. As $\overline{\xvec{y}}$ obeys the property~P\Rref{item:P2} stated in \Cref{subsection:flushing}, the solution to
		\begin{equation}\label{equation:EulerLinearizedNonlocalized}
			\partial_t V + (\overline{\xvec{u}} \cdot \xnab) V = \mathbb{I}_{[T_b,1]}(t)g,
			\quad V(\cdot, 0) = 0
		\end{equation}
		satisfies $V(\cdot, 0) = V(\cdot, T_b)$ regardless the choice of $g$. Thus, applying \Cref{lemma:FiniteDimensionalControlsTransport}  to \eqref{equation:EulerLinearizedNonlocalized} with $T = 1 -T_b$ provides a finite-dimensional control $g \in \xLtwo((0,1); \mathcal{H}(\mathcal{K}))$ such that the solution $V \in \xCzero([0,1];\xHn{{m}})\cap\xWn{{1,2}}((0,1);\xHn{{m-1}})$
		to the transport equation \eqref{equation:EulerLinearizedNonlocalized} satisfies
		\begin{equation}\label{equation:approxctrl_tildev}
			\|V(\cdot,1)-v_1\|_{m} < \varepsilon.
		\end{equation}
		A control $g$ having these properties is now fixed. In particular, recalling the definition of $\mathcal{H}(\mathcal{K})$ from \eqref{equation:HK}, one can represent~$g(\cdot, t)$ for each $t \in [0,1]$ by means of the linear combination
		\begin{equation}\label{equation:LinearCombination_g}
			g(\xvec{x},t) = \sum_{\xsym{\ell} \in \mathcal{K}} \left[\zeta_{\xsym{\ell}}^s(t) s_{\xsym{\ell}}(\xvec{x}) + \zeta_{\xsym{\ell}}^c(t) c_{\xsym{\ell}}(\xvec{x})\right], \quad \xvec{x} \in \mathbb{T}^2
		\end{equation}
		with uniquely determined coefficients $(\zeta_{\xsym{\ell}}^s,\zeta_{\xsym{\ell}}^c)_{\xsym{\ell}\in\mathcal{K}} \subset \xLtwo((0,1);\mathbb{R})$.

		\paragraph{Step 2. Localization.}
		Given the parameters $(\zeta_{\xsym{\ell}}^s,\zeta_{\xsym{\ell}}^c)_{\xsym{\ell}\in\mathcal{K}}$ determined during the previous step, let $\widehat{v}$ be the solution to the version of \eqref{equation:EulerLinearizedlocalized} where the right-hand side is  the auxiliary force $\widehat{\eta}_{(\zeta_{\xsym{\ell}}^s,\zeta_{\xsym{\ell}}^c)_{\xsym{\ell}\in\mathcal{K}}}$ defined in \eqref{equation:ExplicitControl3}, that is
		\begin{equation}\label{equation:tempwithcontroletahat}
			\partial_t \widehat{v} + (\overline{\xvec{y}} \cdot \xnab) \widehat{v} = \widehat{\eta}_{(\zeta_{\xsym{\ell}}^s,\zeta_{\xsym{\ell}}^c)_{\xsym{\ell}\in\mathcal{K}}}, \quad 
			\widehat{v}(\cdot, 0) = 0.
		\end{equation}
		As $\xsym{\mathcal{Y}}$ denotes the flow associated via \eqref{equation:FlowY} with $\overline{\xvec{y}}$, the property P\Rref{item:P2} from \Cref{subsection:flushing}, together with the method of characteristics, gives rise to the representation
		\begin{equation}\label{equation:SolutionFormula_vhat}
			\widehat{v}(\xvec{x},1) = \widehat{v}(\xsym{\mathcal{Y}}(\xvec{x},0,1),1) = \int_0^1 \widehat{\eta}_{(\zeta_{\xsym{\ell}}^s,\zeta_{\xsym{\ell}}^c)_{\xsym{\ell}\in\mathcal{K}}}(\xsym{\mathcal{Y}}(\xvec{x},0,s),s) \, \xdx{s}, \quad \xvec{x} \in \mathbb{T}^2.
		\end{equation}
		Using this representation, we now show $\widehat{v}(\cdot, 1) = V(\cdot, 1)$, which then implies $\|\widehat{v}(\cdot,1)-v_1\|_{m} < \varepsilon$ via \eqref{equation:approxctrl_tildev}.
		Hereto, we take any $\xvec{x} \in \mathbb{T}^2$ and write $\xvec{q} \coloneq \xsym{\mathcal{U}}(\xvec{x},0,1)$, noting that the flow $\xsym{\mathcal{U}}$ is defined in \eqref{equation:flowU}. 
		Employing the partition of unity $(\mu_l)_{l \in \{1,\dots,K\}}$ from \eqref{equation:TranslationsCutoff}, and the solution formula for \eqref{equation:EulerLinearizedNonlocalized}, it follows that
		\begin{equation}\label{equation:SolutionFormula_vtilde}
			\begin{aligned}
				V(\xvec{q},1) & = \int_0^1 \mathbb{I}_{[T_b, 1]}(s)	g(\xsym{\mathcal{U}}(\xvec{x},0,s),s) \, \xdx{s} = \int_{T_b}^1 g(\xsym{\mathcal{U}}(\xvec{x},0,s),s) \, \xdx{s}\\
				& = \sum_{l=1}^K\int_{T_b}^1 \mu_{l}(\xvec{q}) 	g(\xsym{\mathcal{U}}(\xvec{x},0,s),s) \, \xdx{s}.
			\end{aligned}
		\end{equation}
		In view of \eqref{equation:Definition_taul}, a change of variables yields
			\begin{equation}\label{equation:ChangeOfVariables}
				\begin{aligned}
					&\int_{T_b}^1 \mu_{l}(\xvec{q}) g(\xsym{\mathcal{U}}(\xvec{x},0,s),s) \, \xdx{s} \\
					& \quad = \int_{t^{l}_a}^{t^{l}_b} \mu_l(\xvec{q}) g\left(\xsym{\mathcal{U}}\left(\xvec{x},0,\left(T_b + s-t^{l}_a \right)\right),\left(T_b + s-t^{l}_a\right)\right) \, \xdx{s},\\
					& \quad = \int_{t^{l}_a}^{t^{l}_b} \mu_l(\xvec{q}) g\left(\xsym{\mathcal{U}}\left(\xvec{x},0,\tau_{l}(s)\right),\tau_{l}(s)\right) \, \xdx{s}
				\end{aligned}
			\end{equation}
			for each $l \in \{1,\dots,K\}$. By further utilizing in \eqref{equation:ChangeOfVariables} the property P\Rref{item:P3} of $\overline{\xvec{y}}$ and the corresponding flow $\xsym{\mathcal{Y}}$, while recalling the definition of $\chi$ from \eqref{equation:Definition_chi}, one has
			\begin{multline}\label{equation:UsingPropertyP3ToGet_chi}
				\int_{T_b}^1 \mu_{l}(\xvec{q}) g(\xsym{\mathcal{U}}(\xvec{x},0,s),s) \, \xdx{s} \\ = \int_{t^{l}_a}^{t^{l}_b} \chi(\xsym{\mathcal{Y}}(\xvec{q},0,s))g\left(\xsym{\mathcal{U}}\left(\xvec{x}, 0, \tau_{l}(s)\right), \tau_{l}(s)\right) \, \xdx{s}
			\end{multline}
			for each $l \in \{1,\dots,K\}$; see also (\cf~\Cref{Figure:AnalyzingCharacteristics}). Thus, a combination of \eqref{equation:SolutionFormula_vtilde} and \eqref{equation:UsingPropertyP3ToGet_chi} leads to
			\begin{equation}\label{equation:SneakInU}
				\begin{aligned}
					V(\xvec{q},1) = \sum_{l=1}^K \int_{t^{l}_a}^{t^{l}_b} \chi(\xsym{\mathcal{Y}}(\xvec{q},0,s))g\left(\xsym{\mathcal{U}}\left(\xvec{x}, 0, \tau_{l}(s)\right), \tau_{l}(s)\right) \, \xdx{s}.
				\end{aligned}
			\end{equation}
			Since $\xsym{\mathcal{U}}(\cdot,s,t)$ is for each $s,t \in [0,1]$ a homeomorphism of $\mathbb{T}^2$, and by employing \cref{equation:SneakInU,equation:DefinitionPhiPsi,equation:LinearCombination_g}, we infer that
			\begin{equation}\label{equation:TransformationsAlongCharacertistics}
				\begin{aligned}
					V(\xvec{x},1) = \int_0^1 \widehat{\eta}_{(\zeta_{\xsym{\ell}}^s,\zeta_{\xsym{\ell}}^c)_{\xsym{\ell}\in\mathcal{K}}}(\xsym{\mathcal{Y}}(\xvec{x},0,s),s) \, \xdx{s}
				\end{aligned}
			\end{equation}
			for all $\xvec{x} \in \mathbb{T}^2$. Owing to the representation \eqref{equation:SolutionFormula_vhat}, it follows that $\widehat{v}(\cdot, 1) = V(\cdot, 1)$, and thus $\|\widehat{v}(\cdot,1)-v_1\|_{m} < \varepsilon$. 
		
		\begin{figure}[ht!]
			\centering
			\resizebox{0.6\textwidth}{!}{
				\begin{tikzpicture}
					\clip(-0.4,-0.8) rectangle (5.8,3.5);
					
					\draw[line width=0.2mm, color=darkgray] plot[smooth cycle] (0,0) rectangle (0.6,0.6);
					\draw[line width=0.2mm, color=darkgray] plot[smooth cycle] (0,0.45) rectangle (0.6,1.05);
					\draw[line width=0.2mm, color=darkgray] plot[smooth cycle] (0.45,0) rectangle (1.05,0.6);
					\draw[line width=0.2mm, color=darkgray] plot[smooth cycle] (0.45,0) rectangle (1.5,0.6);
					\draw[line width=0.2mm, color=darkgray] plot[smooth cycle] (0.9,0.45) rectangle (1.5,1.05);
					\draw[line width=0.2mm, color=darkgray] plot[smooth cycle] (0,0.9) rectangle (0.6,1.5);
					\draw[line width=0.2mm, color=darkgray] plot[smooth cycle] (0.9,0.9) rectangle (1.5,1.5);
					\draw[line width=0.2mm, color=darkgray] plot[smooth cycle] (0.45,0.9) rectangle (1.05,1.5);
					
					\draw[line width=0.4mm, color=FireBrick] plot[smooth cycle] (0.45,0.45) rectangle (1.05,1.05);
					\draw[line width=0.4mm, color=FireBrick] plot[smooth cycle] (4.7,2.8) rectangle (5.3,3.4);
					
					\draw[line width=0.2mm] (0.9,0.9) -- (5.16,0.9);
					\draw[->, line width=0.2mm] (5.15,0.9) -- (5.15,3.15);
					
					\coordinate[label=below:\tiny$\xvec{x}$] (A) at (0,-0.3);
					\coordinate[label=below:\tiny${\xvec{q}=\xsym{\mathcal{U}}(\xvec{x},0,1) \in \mathcal{O}_i}$] (B) at (2.1,-0.25);
					\coordinate[label=below:\tiny${\xsym{\mathcal{Y}}(\xvec{q},0,t^{i}_a) \in \mathcal{O}}$] (C) at (4.85,-0.25);
					\coordinate[label=right:\tiny\color{FireBrick}$\mathcal{O}_i$] (E) at (0.47,0.72);
					\coordinate[label=right:\tiny\color{FireBrick}$\mathcal{O}$] (F) at (4.65,3);
					
					\fill[line width=0mm, color=FireBrick!100, fill=FireBrick!100] plot[smooth cycle] (0.5,2.5) circle (0.1);
					\draw[line width=0.6mm, color=black, fill=FireBrick!100] plot[smooth cycle] (0.9,0.9) circle (0.1);
					\draw[line width=0.6mm, dash dot, color=black, fill=FireBrick!100] plot[smooth cycle] (5.15,3.25) circle (0.1);
					\fill[line width=0mm, color=FireBrick!100, fill=FireBrick!100] plot[smooth cycle] (-0.3,-0.5) circle (0.1);
					\draw[line width=0.6mm, color=black, fill=FireBrick!100] plot[smooth cycle] (0.79,-0.5) circle (0.1);
					\draw[line width=0.6mm, dash dot, color=black, fill=FireBrick!100] plot[smooth cycle] (3.8,-0.5) circle (0.1);
					
					\draw [->, line width=0.2mm,
					line join=round,
					decorate, decoration={
						zigzag,
						segment length=8,
						amplitude=2.9,post=lineto,
						post length=4pt
					}]  (0.6,2.45) -- (0.88,1.02);
				\end{tikzpicture}
			}
			\caption{A sketch of several ideas related to the identities \eqref{equation:ChangeOfVariables} and \eqref{equation:UsingPropertyP3ToGet_chi}, referring to \Cref{subsection:opencoveringoverlapping} for the notations. In the particular example displayed here, the point $\xvec{q} \in \mathcal{O}_i$ is located on the boundary of $\mathcal{O}_i\cap\mathcal{O}_{l}$ for some $l \in \{1,\dots,K\}\setminus\{i\}$. Thus, the integral over $[T_b,1]$ of $g(\xsym{\mathcal{U}}(\xvec{x},0,\cdot),\cdot)$ can be compressed into an integral over $(t^i_a, t^i_b)$. During this short interval, the square $\mathcal{O}_i$ has already been moved by $\xsym{\mathcal{Y}}$ into $\mathcal{O}$ and the flow $\xsym{\mathcal{Y}}$ pauses. The overlapping squares at the bottom left indicate all members of the family $(\mathcal{O}_l)_{l\in\{1,\dots,K\}}$ which intersect~$\mathcal{O}_i$.}
			\label{Figure:AnalyzingCharacteristics}
		\end{figure}
		
		\paragraph{Step 3. Using an average-free control.} It remains to show that replacing the control $\widehat{\eta}$ from \eqref{equation:ExplicitControl3} by $\widetilde{\eta}$ from \eqref{equation:ExplicitControl2} leads to an average-free trajectory as desired. To this end, we define
		\begin{equation}\label{equation:v}
			v(\xvec{x}, t) \coloneq \widehat{v}(\xvec{x}, t) - \frac{{\chi}(\xvec{x})\int_0^t \int_{\mathbb{T}^2} \widehat{\eta}(\xvec{z}, s) \, \xdx{\xvec{z}} \xdx{s}}{\int_{\mathbb{T}^2} \chi(\xvec{z}) \, \xdx{\xvec{z}}}. 
		\end{equation}
		Clearly, it holds $v(\cdot, 0) = \widehat{v}(\cdot, 0) = 0$ and $v$ solves the controlled transport equation in \eqref{equation:EulerLinearizedlocalized}. Moreover, we shall see next that also
		\begin{equation}\label{equation:propac}
			v(\cdot, 1) = \widehat{v}(\cdot, 1).
		\end{equation}
		Indeed, since $g \in \xLtwo((0,1);\mathcal{H}(\mathcal{K}))$, we recognize that $\smallint_{\mathbb{T}^2} g(\xvec{x},t) \, \xdx{\xvec{x}} = 0$ for almost all $t \in [0,1]$, 
		and because the vector field $\overline{\xvec{u}}$ is divergence-free, we can further infer
		that the zero average of the initial data is preserved during the evolution governed by the transport equation \eqref{equation:EulerLinearizedNonlocalized}. Hence, $\smallint_{\mathbb{T}^2}V(\xvec{x},t) \, \xdx{\xvec{x}} = 0$ for all $t \in [0,1]$ and, since~$\overline{\xvec{y}}$ is divergence-free, it follows from \eqref{equation:TransformationsAlongCharacertistics} that
		\begin{equation*}\label{equation:timeavgctrl}
			0 = \int_{\mathbb{T}^2} V(\xvec{z},1) \, \xdx{\xvec{z}} = \int_{\mathbb{T}^2} \! \int_0^1 \widehat{\eta}(\xsym{\mathcal{Y}}(\xvec{z},0,s),s) \, \xdx{s} \, \xdx{\xvec{z}} = \int_0^1 \!\int_{\mathbb{T}^2} \widehat{\eta}(\xvec{z},s) \, \xdx{\xvec{z}} \, \xdx{s},
		\end{equation*}
		which together with \eqref{equation:v} implies \eqref{equation:propac}. Thus, considering the construction of $v$ and the equation \eqref{equation:tempwithcontroletahat} for $\widehat{v}$, the function~$v$ solves the problem \eqref{equation:EulerLinearizedlocalized} and obeys~\eqref{equation:LinApproxEst}. 
	\end{proof}

	\section{Navier--Stokes driven by scaled transport controls}\label{section:Navier--Stokes}
	Inspired by \cite{Coron96} and \cite{Nersesyan2021}*{Proposition 2.2}, we translate the approximate controllability result provided by \Cref{theorem:localized} to the nonlinear Navier--Stokes system via hydrodynamic scaling. Hereto, a well-posedness result is recalled in \Cref{subsection:wellp} and the scaling argument is carried out in \Cref{subsection:PassingToTheNonlinearSystem}.

	\subsection{Well-posedness}\label{subsection:wellp}
	We write $\xwcurl{\xvec{g}} \coloneq \partial_1 g_2 - \partial_2 g_1$ for the curl of~$\xvec{g} = [g_1, g_2]^{\top}$. Then, for any $m \in \mathbb{N}$, let $\xsym{\Upsilon}\colon \xHn{m}\times\mathbb{R}^2 \longrightarrow \bVn{{m+1}}$ be the operator assigning to $z \in \xHn{m}$ and $\xvec{A} \in \mathbb{R}^2$ the unique vector field $\xvec{g} = \xsym{\Upsilon}(z,\xvec{A}) \in \bVn{{m+1}}$ satisfying
	\begin{equation}\label{equation:divcurlprob}
		\xwcurl{\xvec{g}} = z, \quad \xdiv{\xvec{g}} = 0, \quad \int_{\mathbb{T}^2} \xvec{g}(\xvec{x}) \, \xdx{\xvec{x}} = \xvec{A}.
	\end{equation}
	The operator $\xsym{\Upsilon}$ can be expressed via
	\[
		\xsym{\Upsilon}(z,\xvec{A}) = [\partial_2 \varphi, -\partial_1 \varphi]^{\top} + \xvec{A},
	\]
	where the stream function $\varphi \in \xHn{{m+2}}$ is the unique solution to the Poisson problem
	\[
		- \Delta \varphi = z, \quad  \int_{\mathbb{T}^2} \varphi(\xvec{x}) \, \xdx{\xvec{x}} = 0.
	\]	
	By resorting to the elliptic theory for the div-curl problem \eqref{equation:divcurlprob}, we fix any constant~$C_0 > 0$ such that
	\begin{equation}\label{equation:velocityvorticityestimate}
		\|\xsym{\Upsilon}(z,\xvec{A})\|_{m+1} \leq C_0(\|z\|_m + |\xvec{A}|)
	\end{equation}
	for all $z \in \xHn{m}$ and $\xvec{A} \in \mathbb{R}^2$. 
	
	We proceed with a few remarks concerning the global well-posedness of the two-dimensional Navier--Stokes system. Here, for any given time $T > 0$ and a regularity parameter $m \in \mathbb{N}$, the solution space for the vorticity is chosen as
	\[
		\mathcal{X}_{T}^m \coloneq \xCzero([0,T]; \xHn{m}) \cap \xLtwo((0,T); \xHn{{m+1}}),
	\]
	endowed with the norm $\|\cdot\|_{\mathcal{X}_{T}^m} \coloneq \|\cdot\|_{\xCzero([0,T]; \xHn{m})} + \|\cdot\|_{\xLtwo((0,T); \xHn{{m+1}})}$.
	
	\begin{lmm}\label{lemma:Wellposedness}
		For any $w_0 \in \xHn{{m}}$, $\xsym{A} \in \xWn{{1,2}}((0,T);\mathbb{R}^2)$, and $h \in \xLtwo((0,T);\xHn{{m-1}})$, there exists a unique solution $w = {S}_{T}(w_0,h,\xsym{A}) \in \mathcal{X}_{T}^m$
		to the vorticity equation 
		\begin{equation}\label{equation:GeneralNavierStokesOrEuler}
			\partial_t w - \nu \Delta w + \left(\xsym{\Upsilon}(w,\xsym{A}) \cdot \xnab\right) w = h, \quad
			w(\cdot, 0) = w_0,
		\end{equation}
		where the resolving operator for \eqref{equation:GeneralNavierStokesOrEuler} is denoted by
		\[
			{S}_{T}(\cdot, \cdot, \cdot)\colon \xHn{m}\times\xLtwo((0,T);\xHn{{m-1}})\times\xWn{{1,2}}((0,T);\mathbb{R}^2) \longrightarrow \mathcal{X}_{T}^m, \quad (w_0,h,\xsym{A}) \longmapsto w.
		\]
	\end{lmm}
	\begin{proof}
		Given any $\xvec{f} \in \xLtwo((0,T);\bVn{m})$ with $\xwcurl{\xvec{f}} = h$, one may reduce the well-posedness and regularity statements of \Cref{lemma:Wellposedness} to corresponding ones for the two-dimensional Navier--Stokes problem in velocity form
		\begin{equation}\label{equation:navsavcor}
			\partial_t \xvec{u} - \nu \Delta \xvec{u} + \left(\xvec{u} \cdot \xnab\right) \xvec{u} + \xnab p = \xvec{f} + \partial_t \xvec{a},\quad
			\xdiv{\xvec{u}} = 0,\quad
			\xvec{u}(\cdot, 0) =  \xsym{\Upsilon}(w_0, \xsym{A}(0)),
		\end{equation}
		where $\xvec{a}(t) \coloneq \xsym{A}(t) - \int_0^t \int_{\mathbb{T}^2} \xvec{f}(\xvec{x},s)  \, \xdx{\xvec{x}}  \, \xdx{s}$ is chosen such that the solution $\xvec{u}(t)$ to \eqref{equation:navsavcor} has average $\xsym{A}(t)$. 
		For the well-posedness of \eqref{equation:navsavcor}, we refer, \eg, to \cite{Temam2001,ConstantinFoias1988}.
	\end{proof}
	
	\subsection{Large controls in small time}\label{subsection:PassingToTheNonlinearSystem}
	Via hydrodynamic scaling arguments, the following lemma translates controls for linear transport to vorticity controls. Here, $\overline{\xvec{y}}$ is the vector field fixed in \Cref{subsection:flushing}.
	
	\begin{lmm}\label{lemma:SmallTimeConvergenceToFinalstateOfLinearized}
		Let any $m \in \mathbb{N}$, initial state $w_0 \in \xHn{{m+1}}$, forces $g, h \in \xLtwo((0,1);\xHn{{m-1}})$, and velocity average $\widetilde{\xvec{U}} \in \xWn{{1,2}}((0,1); \mathbb{R}^2)$ be fixed. Furthermore, let~$v$ be the solution to the linear transport problem
		\begin{equation}\label{equation:InitialStateEulerLinearizedlocalized}
			\partial_t v + (\overline{\xvec{y}} \cdot \xnab) v = g, \quad
			v(\cdot, 0) = w_0,
		\end{equation}
		and denote
		\begin{gather*}
			g_{\delta}(\cdot, t) \coloneq \delta^{-1}g(\cdot, \delta^{-1}t), \quad \overline{\xvec{y}}_{\delta}(t) \coloneq \delta^{-1} \overline{\xvec{y}}(\delta^{-1}t), \\
			\widetilde{\xvec{U}}_{\delta}(t) \coloneq \widetilde{\xvec{U}}(\delta^{-1}t), \quad \widetilde{\xsym{\aleph}}_{\delta}(t) \coloneq \overline{\xvec{y}}_{\delta}(t) + \widetilde{\xvec{U}}_{\delta}(t)			
		\end{gather*}
		for $t \in (0,\delta)$. Then, one has the convergence
		\begin{equation}\label{equation:smalltimelimit_to_finalstate_of_linearized}
			\|S_{\delta}(w_0, h + g_{\delta}, \widetilde{\xsym{\aleph}}_{\delta})|_{t=\delta} - v(\cdot,1)\|_m \longrightarrow 0 \quad \mbox{as } \delta \longrightarrow 0,
		\end{equation}
		uniformly in $w_0$ and $g,h$ from bounded subsets of $\xHn{{m+1}}$ and $\xLtwo((0,1);\xHn{{m-1}})$ respectively.
	\end{lmm}
	\begin{proof}
		To emphasize that \eqref{equation:smalltimelimit_to_finalstate_of_linearized} will be uniform with respect to data from bounded subsets, we fix any~$M_0 > 0$ with
		\begin{equation}\label{equation:M_0Bound}
			\|w_0\|_{m+1} + \|g\|_{\xLtwo((0,1);\xHn{{m-1}})} + \|h\|_{\xLtwo((0,1);\xHn{{m-1}})} \leq M_0.
		\end{equation}
		When $w_0$ and $g$ satisfy \eqref{equation:M_0Bound}, the solution $v$ to the transport problem \eqref{equation:InitialStateEulerLinearizedlocalized} remains in a bounded set $\xB$ in $\xCzero([0,1];\xHn{{m+1}})\cap \xWn{{1,2}}((0,1);\xHn{m})$, where $\xB$ depends only on $M_0$. 
		Moreover, we select $w = S_{\delta}(w_0, h + g_{\delta}, \widetilde{\xsym{\aleph}}_{\delta})$ with associated velocity~$\xvec{u} \coloneq \xsym{\Upsilon}(w,\widetilde{\xsym{\aleph}}_{\delta})$.
		
		\paragraph{Step 1. Ansatz.} We consider for any given $\delta \in (0,1)$ the velocity and vorticity expansions
		\begin{equation}\label{equation:ansatz}
			w = v_{\delta} + r, \quad \xvec{u} = \overline{\xvec{y}}_{\delta} + \xvec{V}_{\delta} + \xvec{R}, \quad \xvec{V}_{\delta} \coloneq \xsym{\Upsilon}(v_{\delta}, \widetilde{\xvec{U}}_{\delta}), \quad v_{\delta}(\cdot,t) \coloneq v(\cdot, \delta^{-1}t)
		\end{equation}
		with remainders
		\[
			r\colon \mathbb{T}^2\times[0,\delta] \longrightarrow \mathbb{R}, \quad \xvec{R} \coloneq \xsym{\Upsilon}(r, \xsym{0}).
		\]
		Due to the property~P\Rref{item:P1} of $\overline{\xvec{y}}$ from \Cref{theorem:returnmethodflow} and the choice of $\widetilde{\xsym{\aleph}}_{\delta}$, the ansatz for $\xvec{u}$ in \eqref{equation:ansatz} is consistent with the definitions of $\overline{\xvec{y}}_{\delta}$, $\xvec{V}_{\delta}$, and $\xvec{R}$. Furthermore, the initial condition $v(\cdot,0) = w_0$ in \eqref{equation:InitialStateEulerLinearizedlocalized} implies that $r(\cdot,0) = 0$ in $\mathbb{T}^2$. Therefore, we can show \eqref{equation:smalltimelimit_to_finalstate_of_linearized} by verifying
		\begin{equation}\label{equation:MilestoneSmallTimeConvergenceToFinalstateOfLineraized}
			\|r(\cdot, \delta)\|_{m} \longrightarrow 0 \, \mbox{ as } \, \delta \longrightarrow 0.
		\end{equation}
		The convergence in \eqref{equation:MilestoneSmallTimeConvergenceToFinalstateOfLineraized} shall be demonstrated next.
		\paragraph{Step 2. Energy estimates for the remainder.}
		After plugging \eqref{equation:ansatz} and the definition of~$v_{\delta}$ into the equation satisfied by $S_{\delta}(w_0, h + g_{\delta})$, which is of the form \cref{equation:ThmNavierSTokesOrEuler}, the remainder $r$ is seen to obey in $\mathbb{T}^2 \times (0,\delta)$ the initial value problem
		\begin{equation}\label{equation:RemainderEvolutionequation}
			\partial_t r - \nu \Delta r + ((\overline{\xvec{y}}_{\delta} + \xvec{V}_{\delta} + \xvec{R}) \cdot \xnab) r + (\xvec{R} \cdot \xnab)v_{\delta} = \xi_{\delta}, \quad
			r(\cdot, 0) = 0,	
		\end{equation}
		with $\xi_{\delta} \coloneq h - (\xvec{V}_{\delta} \cdot \xnab)v_{\delta} + \nu \Delta v_{\delta}$. In particular, for all $t \in [0,\delta]$ one has
		\begin{equation}\label{equation:EllipticregularityBoundsRandVdelta}
			\|\xvec{R}(\cdot,t)\|_{m+1} \leq C_0 \| r(\cdot,t) \|_{m}, \quad \|\xvec{V}_{\delta}(\cdot,t)\|_{m+1} \leq C_0 \left(\| v_{\delta}(\cdot,t)\|_{m} + |\widetilde{\xvec{U}}_{\delta}(t)|\right),
		\end{equation}
		where $C_0 > 0$ is the constant from \eqref{equation:velocityvorticityestimate}.
		
		We proceed by taking the $\xLtwo(\mathbb{T}^2)$-inner product of \eqref{equation:RemainderEvolutionequation} with $\Delta^m r$, followed by integrating in time from~$0$ to~$\delta$. After integration by parts, while also applying the Poincar\'e inequality and noting that $\xdiv{\xvec{R}} = 0$, one obtains for $t \in (0,\delta)$ the estimate
		\begin{multline*}\label{EnergyequalityRemainder}
			\|r(\cdot,t)\|_m^2 + 2\nu \int_0^t \|r(\cdot,s)\|_{m+1}^2 \, \xdx{s} \\
			\begin{aligned}
				& \leq C\int_0^t \|\xi_{\delta}(\cdot, s)\|_{m-1}\|r(\cdot,s)\|_{m+1} \, \xdx{s}\\
				& \quad + 2\int_0^t \|\xvec{R}(\cdot, s)\|_{m+1}\left(\|v_{\delta}(\cdot, s)\|_{m+1}\|r(\cdot, s)\|_m + \|r(\cdot, s)\|_m \|r(\cdot, s)\|_{m+1}\right)  \, \xdx{s} \\
				& \quad + 2\int_0^t \|\overline{\xvec{y}}_{\delta}(s) + \xvec{V}_{\delta}(\cdot, s)\|_{m+1} \|r(\cdot, s)\|_m^2  \, \xdx{s}\\
				& = I_1 + I_2 + I_3.
			\end{aligned}
		\end{multline*}
		Regarding $I_1$, integration by parts, combined with inequalities of Young and Cauchy-Schwarz, yields for a constant $C = C(\nu) > 0$ the estimate
		\begin{equation}\label{equation:I_1_a}
			\begin{aligned}
				I_1 & \leq C\int_0^{t} \|\xi_{\delta}(\cdot, s)\|_{m-1}^2  \, \xdx{s} + \frac{\nu}{4}\int_0^t \|r(\cdot,s)\|_{m+1}^2 \, \xdx{s} \\
				& \leq C \int_0^t \left(\|h(\cdot, s)\|_{m-1}^2 + \|(\xvec{V}_{\delta}(\cdot, s) \cdot \xnab)v_{\delta}(\cdot, s)\|_{m-1}^2 + \|\nu \Delta v_{\delta}(\cdot, s)\|_{m-1}^2\right) \, \xdx{s} \\
				& \quad + \frac{\nu}{4}\int_0^t \|r(\cdot,s)\|_{m+1}^2 \, \xdx{s}.
			\end{aligned}
		\end{equation}
		Consequently, by a change of variables in the time integrals of \eqref{equation:I_1_a} and taking into account that $t \in (0, \delta)$, one has the bound
		\begin{equation*}\label{equation:I_1_b}
			\begin{aligned}
				I_1 & \leq C \int_0^{\delta} \|h(\cdot, s)\|_{m-1}^2 \, \xdx{s} + C \delta \int_0^1 \|(\xvec{V}_{\delta}(\cdot, \delta s) \cdot \xnab)v(\cdot, s)\|_{m-1}^2 \, \xdx{s} \\
				& \quad + \delta \nu^2 C \int_0^1 \|v(\cdot, s)\|_{m+1}^2 \, \xdx{s} + \frac{\nu}{4}\int_0^t \|r(\cdot,s)\|_{m+1}^2 \, \xdx{s}.
			\end{aligned}
		\end{equation*}
		Next, by utilizing the estimate for $\xvec{R}$ from \eqref{equation:EllipticregularityBoundsRandVdelta}, the integral~$I_2$ is bounded via
		\begin{equation*}
			\begin{aligned}
				I_2 & \leq 2C_0 \int_0^t \|r(\cdot, s)\|_{m}\left(\|v_{\delta}(\cdot, s)\|_{m+1}\|r(\cdot, s)\|_m + \|r(\cdot, s)\|_m\|r(\cdot, s)\|_{m+1}\right) \, \xdx{s} \\
				& \leq C_0\int_0^t \|v_{\delta}(\cdot, s)\|_{m+1}^2 \, \xdx{s} + \frac{\nu}{4}\int_0^t \|r(\cdot, s)\|_{m+1}^2 \, \xdx{s} + C \int_0^t \|r(\cdot, s)\|_m^4 \, \xdx{s},
			\end{aligned}
		\end{equation*}
		and a change of variables implies 
		\[
			\int_0^t \|v_{\delta}(\cdot, s)\|_{m+1}^2 \, \xdx{s} \leq \delta \int_0^1  \|v(\cdot, s)\|_{m+1}^2 \, \xdx{s} \longrightarrow 0  \quad \mbox{as } \delta \longrightarrow 0.
		\]
		Concerning $I_3$, we denote $M \coloneq \sup_{s \in [0,1]}\overline{\xvec{y}}(s)$, which is independent of $\delta \in (0,1)$, and perform a change of variables in order to get
		\begin{equation*}
			\begin{aligned}
				\widetilde{I}_{3,\delta}(t) & \coloneq \int_0^t \|\overline{\xvec{y}}_{\delta}(s) + \xvec{V}_{\delta}(\cdot, s)\|_{m+1}  \, \xdx{s} \\
				& \leq \int_0^1 \left(M + \delta C_0 \|v(\cdot, s)\|_{m} + \delta C_0 |\widetilde{\xvec{U}}(s)|\right) \, \xdx{s}  \longrightarrow M \quad \mbox{as } \delta \longrightarrow 0.
			\end{aligned}
		\end{equation*}
		Then, by Gr\"onwall's inequality, for each $\delta \in (0,1)$ there exists $\varepsilon_{\delta} = \varepsilon_{\delta}(M_0) > 0$, which is independent of $t \in [0,\delta]$ and data obeying \eqref{equation:M_0Bound}, such that $\lim\limits_{\delta\to0}\varepsilon_{\delta} = 0$ and
		\begin{equation*}
			\|r(\cdot, t)\|_m^2 \leq \left( \varepsilon_{\delta} + C\int_0^t \|r(\cdot,s)\|_m^4 \, \xdx{s} \right) \exp \left(\widetilde{I}_{3,\delta}(t) \right),
		\end{equation*}
		where $C > 0$ is independent of $\delta$.
		Thus, there exists a new constant $C > 0$, which is likewise independent of $\delta \in (0,1)$, $t \in [0,\delta]$, and data satisfying \eqref{equation:M_0Bound}, such that  
		\begin{equation*}
			\|r(\cdot, t)\|_m^2 \leq C\varepsilon_{\delta} + C\int_0^t \|r(\cdot,s)\|_m^4 \, \xdx{s} \eqcolon \Theta(t).
		\end{equation*}
		As a result, by comparison with the differential inequality $\Theta' \leq C \Theta^2$  (\cf~\cite{Nersesyan2021}*{Proposition 2.2}), we conclude \eqref{equation:MilestoneSmallTimeConvergenceToFinalstateOfLineraized} and thus arrive at~\eqref{equation:smalltimelimit_to_finalstate_of_linearized}.
	\end{proof}

	\section{Proof of the main result}\label{section:General}
	This section is devoted to proving \Cref{theorem:main0_vel}. Hereto, we fix $k \in \mathbb{N}$, the control time~$T_{\operatorname{ctrl}} > 0$, the states $\xvec{u}_0, \xvec{u}_1 \in \bVn{{k+1}}$, the force $\xvec{f} \in \xLtwo((0,T_{\operatorname{ctrl}});\xHn{{k+1}}(\mathbb{T}^2;\mathbb{R}^2))$, and the viscosity $\nu > 0$; see \Cref{remark:energysolution} for the regularity assumptions in this section. Further, the profile $\overline{\xvec{y}}$ is given by \Cref{theorem:returnmethodflow}, and, for simplicity of the presentation, we assume that the number $N$ in \eqref{equation:IntroductionLocalizedControl2} is even and occasionally identify $\mathcal{K}$, when used as an index set, with $\{1,\dots, N/2\}$ or $\{N/2+1,\dots N\}$.

	\subsection{Vorticity formulation of the controllability problem}\label{subsection:vorticity_formulation}
	Given initial and target states $w_0$ and $w_1$, an accuracy parameter $\varepsilon > 0$, and a prescribed force $h$, we consider for the vorticity $w = \xwcurl{\xvec{u}}\colon\mathbb{T}^2\times(0,T_{\operatorname{ctrl}}) \longrightarrow \mathbb{R}$ the controllability problem
	\begin{equation}\label{equation:00NavierSTokesOrEuler}
			\partial_t w - \nu \Delta w + \left(\xsym{\Upsilon}(w, \xsym{\aleph}) \cdot \xnab\right) w = h + \mathbb{I}_{\omegaup}\eta, \quad
			w(\cdot, 0) = w_0,
	\end{equation}
	where the finitely decomposable control $\eta$ and the velocity average~$\xsym{\aleph} = \xsym{\aleph}_{\sigma}$ are sought such that
	\begin{equation*}\label{equation:terminalconstrnonlvorticityform}
		\|w(\cdot, T_{\operatorname{ctrl}}) - w_1\|_{k} < \varepsilon.
	\end{equation*}
	As a result of the subsequent analysis, the average $\xsym{\aleph} = \xsym{\aleph}_{\sigma} \in \xWn{{1,2}}((0,T_{\operatorname{ctrl}}); \mathbb{R}^2)$ shall be given in terms of the yet unknown scaling parameter $\sigma > 0$:
	\begin{equation}\label{equation:Tsig}
		\begin{gathered}
			\xsym{\aleph}_{\sigma}(t) \coloneq  \mathbb{I}_{[0, T_{\sigma}]}(t)\xvec{U}(0) +  \mathbb{I}_{[T_{\sigma},T_{\operatorname{ctrl}}]}(t) (\sigma\overline{\xvec{y}}+\xsym{U})(\sigma (t-T_{\sigma})), \quad T_{\sigma} \coloneq T_{\operatorname{ctrl}}-\sigma^{-1},
		\end{gathered}
	\end{equation}
	where $\xvec{U} \in \xWn{{1,2}}((0,1); \mathbb{R}^2)$ is selected such that
	\begin{equation}\label{equation:profileU}
		\forall t \in [0,1/4] \colon \, \xvec{U}(t) = \int_{\mathbb{T}^{2}} \xvec{u}_0(\xvec{x}) \, \xdx{\xvec{x}}, \quad \forall t \in [3/4,1] \colon \, \xvec{U}(t) = \int_{\mathbb{T}^{2}} \xvec{u}_1(\xvec{x}) \, \xdx{\xvec{x}}. 
	\end{equation}
	Moreover, the force $\eta = \eta_{\gamma_1,\dots,\gamma_{N},\sigma}$ will have an explicit representation in terms of the yet unknown control coefficients $\gamma_1$, $\dots$, $\gamma_{N+2}, \sigma$ and the universally fixed profiles $\eta_1, \dots, \eta_{N+2}$, whose choice only depends on the region~$\omegaup$ fixed in \Cref{section:introduction}. Namely,
	\begin{equation}\label{equation:formeta}
		\begin{aligned}
			\eta(\xvec{x},t) & = \sum_{l=1}^N \gamma_l(t) \eta_l(\xvec{x}, \sigma (t-T_{\sigma})) + \sum_{l=N+1}^{N+2}\gamma_{l}(t) \eta_{l}(\xvec{x}),
		\end{aligned}
	\end{equation}
	where
	\begin{equation*}
		\begin{aligned}
			\eta_l(\xvec{x},t) = \begin{cases}
				\chi(\xvec{x}) s_{l} \left(\xsym{\Xi}(\xvec{x},t) \right), & l \in \{1,\dots, N/2\},\\
				\chi(\xvec{x}) c_{l} \left(\xsym{\Xi}(\xvec{x},t) \right), & l \in \{N/2+1,\dots, N\},\\
				(\overline{\xvec{y}}\cdot\xnab) {\chi}(\xvec{x}), & l = N + 1, \\ 
				{\chi}, & l = N + 2
			\end{cases}
		\end{aligned}
	\end{equation*}
	with the cutoff $\chi$ from \eqref{equation:Definition_chi} and the composed flow $\xsym{\Xi}$ defined in \eqref{equation:DefinitionPhiPsi}. Furthermore, we will take $\sigma \geq T_{\operatorname{ctrl}}^{-1}$, $\gamma_1(t) = \dots = \gamma_{N+2}(t) = 0$ for $t \in [0, T_{\sigma}]$, and it will be possible to express $\gamma_{N+1}$ and $\gamma_{N+2}$ through formulas in terms of $\gamma_1, \dots, \gamma_N$.

	\subsection{Controllability of the vorticity formulation}\label{subsection:ctrlvorticity}
	The following theorem states the approximate controllability of the vorticity formulation~\eqref{equation:00NavierSTokesOrEuler} in small time via large controls. Throughout, we denote
	\[
		\widetilde{\xsym{\aleph}}_{\delta}(t) = \delta^{-1}\overline{\xvec{y}}(\delta^{-1}t) + \xvec{U}(\delta^{-1}t), \quad \delta > 0,
	\]
	where $\xvec{U}$ is the profile defined in \eqref{equation:profileU}.
	\begin{thrm}\label{theorem:main_largecontrol}
		Given any $h \in \xLtwo((0,1);\xHn{{k-1}})$ and $w_0, w_1 \in \xHn{{k+1}}$, there exist control parameters $(\zeta_l)_{l \in \{1,\dots,N\}} \subset \xLtwo((0, 1); \mathbb{R})$
		such that the sequence of solutions $(w_{\delta} \in \mathcal{X}_{\delta}^k)_{\delta > 0}$ to the respective vorticity problems
		\begin{equation}\label{equation:ThmNavierSTokesOrEuler}
			\begin{gathered}
				\partial_t w_{\delta} - \nu \Delta w_{\delta} + \left(\xsym{\Upsilon}(w_{\delta},\widetilde{\xsym{\aleph}}_{\delta}) \cdot \xnab\right) w_{\delta} = h + \delta^{-1}\widetilde{\eta}_{\zeta_1,\dots,\zeta_N}(\cdot, \delta^{-1} \cdot), \quad w_{\delta}(\cdot, 0) = w_0
			\end{gathered}
		\end{equation}
		satisfies
		\begin{equation}\label{equation:deltatozero}
			w_{\delta}(\cdot, \delta) \longrightarrow  w_1 \, \mbox{ in } \, \xHn{k} \, \mbox{ as } \,  \delta \longrightarrow 0.
		\end{equation}
		Moreover, $(\zeta_l)_{l \in \{1,\dots,N\}} \subset \xLtwo((0, 1); \mathbb{R})$ can be selected, depending on $w_0, w_1 \in \xHn{{k+1}}$, such that the convergence in \eqref{equation:deltatozero} is uniform with respect to $w_0, w_1$ and $h$ from bounded subsets of $\xHn{{k+1}}$ and $\xLtwo((0,1);\xHn{{k-1}})$.
	\end{thrm}
	\begin{proof} 
		Let $\widetilde{v}$ be the solution in $\mathbb{T}^2 \times [0,1]$ to the linear homogeneous transport equation
		\begin{equation*}\label{equation:FreeEulerLinearizedlocalized}
			\partial_t \widetilde{v} + (\overline{\xvec{y}} \cdot \xnab) \widetilde{v} = 0, \quad
			\widetilde{v}(\cdot, 0) = w_0.
		\end{equation*}
		In particular, we know that $\widetilde{v}(1) = w_0$, as the flow $\xsym{\mathcal{Y}}$ associated with the spatially constant vector field $\overline{\xvec{y}}$ satisfies the property {\rm P}\Rref{item:P2} from \Cref{subsection:flushing}. Now, we fix any~$\varepsilon > 0$ and apply \Cref{theorem:localized} with the final state $\widehat{v}_1 \coloneq w_1 - w_0$, in order to obtain controls $(\zeta_l)_{l \in \{1,\dots,N\}} = (\zeta_{\xsym{\ell}}^s,\zeta_{\xsym{\ell}}^c)_{\xsym{\ell}\in\mathcal{K}} \subset \xLtwo((0,1);\mathbb{R})$ such that the corresponding solution 
		\[
			\widehat{v} \in \xCzero([0,1];\xHn{{k+1}})\cap\xWn{{1,2}}((0,1);\xHn{k})
		\]
		to the transport equation \eqref{equation:EulerLinearizedlocalized} obeys $\|\widehat{v}(\cdot,1) - \widehat{v}_1\|_{k+1} < \varepsilon$.
		Accordingly, the superposition $v \coloneq \widetilde{v} + \widehat{v}$ solves 
		an initial value problem of the type \eqref{equation:InitialStateEulerLinearizedlocalized} with right-hand side $\widetilde{\eta}_{\zeta_1,\dots,\zeta_N}$ and satisfies $\|v(\cdot,1) - w_1\|_{k+1} < \varepsilon$.
		Thus, an application of \Cref{lemma:SmallTimeConvergenceToFinalstateOfLinearized} provides for sufficiently small $\delta > 0$ the estimate
		\[
			\|S_{\delta}\left(w_0, h + \delta^{-1} \widetilde{\eta}_{\zeta_1,\dots,\zeta_N}(\cdot,\delta^{-1}\cdot), \widetilde{\xsym{\aleph}}_{\delta} \right)\!\!|_{t=\delta} - w_1\|_{k} < \varepsilon,
		\]
		where $\delta$ can be chosen uniformly with respect to $w_0, w_1$ and $h$ from respective bounded subsets of $\xHn{{k+1}}$ and $\xLtwo((0,1);\xHn{{k-1}})$. Indeed, as explained in \Cref{remark:rinv}, the family $(\zeta_l)_{l \in \{1,\dots,N\}}$ can be recovered from the initial and target states by means of a continuous linear operator. Hence, the controls $(\zeta_l)_{l \in \{1,\dots,N\}}$, and thus $\widetilde{\eta}_{\zeta_1,\dots,\zeta_N}$, remain in a bounded subset of $\xLtwo((0,1);\mathbb{R})$, respectively $\xLtwo((0,1);\xHn{{k-1}})$, when $w_0$ and $w_1$ vary in a bounded subset of $\xHn{{k+1}}$, so that the uniform choice of $\delta$ is guaranteed by \Cref{lemma:SmallTimeConvergenceToFinalstateOfLinearized}.
	\end{proof}
	
	The next theorem is concerned with the approximate controllability of \eqref{equation:00NavierSTokesOrEuler} on the original time interval $[0, T_{\operatorname{ctrl}}]$, thereby constituting a vorticity version of \Cref{theorem:main0_vel}. 
	\begin{thrm}\label{theorem:main0_general}
		Let $T_{\operatorname{ctrl}} > 0$, $h \in \xLtwo((0,T_{\operatorname{ctrl}});\xHn{{k}})$, $w_0, w_1 \in \xHn{{k}}$, and $\varepsilon > 0$ be arbitrary. There exist $\sigma \geq T_{\operatorname{ctrl}}^{-1}$ and control parameters
		\[
			(\gamma_l)_{l \in \{1,\dots,N\}}  \subset 	\xLtwo((0, 1); \mathbb{R})
		\]
		such that the unique solution $w \in \mathcal{X}_{T_{\operatorname{ctrl}}}^k$ to \eqref{equation:00NavierSTokesOrEuler}, where $\eta$ is obtained from $(\gamma_l)_{l \in \{1,\dots,N\}}$ via \eqref{equation:formeta},
		satisfies the terminal condition
		\begin{equation}\label{equation:rcondthm}
			\|w(\cdot, T_{\operatorname{ctrl}}) - w_1\|_{k} < \varepsilon.
		\end{equation}
	\end{thrm}

	\begin{proof}
		The idea is to activate the control only shortly before reaching the terminal time~$T_{\operatorname{ctrl}}$ (\cf~\Cref{Figure:activinactivecontr}). However, as $\eta$ should be exactly of the form~\eqref{equation:formeta} with $N+3$ unknown parameters $\gamma_1, \dots, \gamma_{N+2}$, and $\sigma$, the main difficulty is now to select via \Cref{theorem:main_largecontrol} a value of $\sigma^{-1} = \delta > 0$ which allows switching on the control at time $T_{\operatorname{ctrl}}-\delta$ while reaching the $\varepsilon$-neighborhood in $\xHn{k}$ of $w_1$ at time $t = T_{\operatorname{ctrl}}$. An alternative strategy, which is more straightforward but requires one control parameter more, is sketched in \Cref{remark:delta0delta} below. Throughout, the initial velocity average is abbreviated by $\xvec{U}_0 \coloneq \int_{\mathbb{T}^2} \xvec{u}_0(\xvec{x}) \, \xdx{\xvec{x}}$.
		
		\paragraph{Step 1. Determining $\delta > 0$.}
		We fix $\widetilde{w}_1 \in \xHn{{k+1}}$ such that $\|w_1 -\widetilde{w}_1\|_{k} < \varepsilon/2$, which is possible by density. Moreover, we denote by $\widetilde{w}_0\coloneq{S}_{T_{\operatorname{ctrl}}}(w_0,h,\xvec{U}_0)_{|_{t = T_{\operatorname{ctrl}}}}$ the state reached when the control is inactive during the whole time interval $[0,T_{\operatorname{ctrl}}]$. Due to $h \in \xLtwo((0,T_{\operatorname{ctrl}});\xHn{{k}})$, \Cref{lemma:Wellposedness} and known parabolic smoothing effects (\cf~\cite{Temam2001}) imply that ${S}_{T_{\operatorname{ctrl}}}(w_0,h,\xvec{U}_0) \in \xCzero((0,T_{\operatorname{ctrl}}];\xHn{{k+1}})$; hence, $\widetilde{w}_0 \in \xHn{{k+1}}$. By applying \Cref{theorem:main_largecontrol} with the initial and target states $\widetilde{w}_0$ and $\widetilde{w}_1$ respectively, while defining $h(\cdot, T_{\operatorname{ctrl}} - s) = 0$ if $s > T_{\operatorname{ctrl}}$, one obtains a small number $\delta \in (0, T_{\operatorname{ctrl}})$ and control parameters $(\rho_{l})_{l \in \{1,\dots,N\}} \subset \xLtwo((0, 1); \mathbb{R})$
		such that the function
		\begin{equation}\label{equation:widetildewdelta}
			\widetilde{w}_{\delta} \coloneq {S}_{\delta}\left(\widetilde{w}_0,h(\cdot,  T_{\operatorname{ctrl}} - \cdot) + \delta^{-1}\widetilde{\eta}_{\rho_1,\dots,\rho_N}(\cdot, \delta^{-1} \cdot), \widetilde{\xsym{\aleph}}_{\delta} \right)
		\end{equation}
		meets the terminal condition
		\begin{equation}\label{equation:wtildetercond}
			\|\widetilde{w}_{\delta}(\cdot, \delta) - \widetilde{w}_1\|_{k} < \varepsilon/2.
		\end{equation}
		Then, since ${S}_{T_{\operatorname{ctrl}}}(w_0,h,\xvec{U}_0) \in \xCzero((0,T_{\operatorname{ctrl}}];\xHn{{k+1}})$, and because \Cref{theorem:main_largecontrol} allows choosing~$\delta$ uniformly with respect to initial data from a bounded subset of $\xHn{{k+1}}$, we can assume~$\delta \in (0, T_{\operatorname{ctrl}})$ to be sufficiently small such that \eqref{equation:wtildetercond} remains valid when replacing~$\widetilde{w}_0$ (and thus $(\rho_{l})_{l \in \{1,\dots,N\}}$) by states from the bounded subset
		\[
			\left\{{S}_{T_{\operatorname{ctrl}}}(w_0,h,\xvec{U}_0)_{|_{t = s}} \, | \, s \in (T_{\operatorname{ctrl}}/2, T_{\operatorname{ctrl}}] \right\} \subset \xHn{{k+1}}.
		\]
		Now, the value of $\delta$ is set and this step is complete. However, for the sake of clarity, let us make two additional remarks.
		
		{\bf First remark.} The state $\widetilde{w}_{\delta}$ and the associated control parameters $(\rho_l)_{l\in\{1,\dots,N\}}$ will not be used explicitly in what follows; they were only introduced as auxiliaries for fixing~$\delta$.
		
		{\bf Second remark.} We inserted $h(\cdot,  T_{\operatorname{ctrl}} - \cdot)$ in \eqref{equation:widetildewdelta} because $\delta$ was unknown at that point; as $h(\cdot,  T_{\operatorname{ctrl}} - \delta + \cdot)$ depends on $\delta$, we could not have employed \Cref{theorem:main_largecontrol} using that force. After $\delta$ is chosen, we can use the correct external force $h(\cdot,  T_{\operatorname{ctrl}} - \cdot)$ without impacting \eqref{equation:wtildetercond}. Indeed, since the convergence property \eqref{equation:deltatozero} in \Cref{theorem:main_largecontrol} is uniform with respect to prescribed external forces from a bounded subset of $\xLtwo((0,1);\xHn{{k-1}})$, and by noticing that
		\begin{equation}\label{equation:hfoba}
			\int_0^{\delta} \|h(\cdot, T_{\operatorname{ctrl}}-s)\|_{k-1}^2 \, \xdx{s}  =  \int_0^{\delta} \|h(\cdot, T_{\operatorname{ctrl}}-\delta + s)\|_{k-1}^2 \, \xdx{s},
		\end{equation}
		one can infer that also
		\begin{equation*}\label{equation:w_hat}
			\widehat{w}_{\delta} \coloneq {S}_{\delta}\left(\widetilde{w}_0,h(\cdot,  T_{\operatorname{ctrl}} - \delta + \cdot) + \delta^{-1}\widetilde{\eta}_{\rho_1,\dots,\rho_N}(\cdot, \delta^{-1} \cdot), \widetilde{\xsym{\aleph}}_{\delta} \right)
		\end{equation*}
		satisfies
		\begin{equation*}\label{equation:w_hat_tc}
			\|\widehat{w}_{\delta}(\cdot, \delta) - \widetilde{w}_1\|_{k} < \varepsilon/2.
		\end{equation*}

		\paragraph{Step 2. Gluing two trajectories.}
		By applying \Cref{theorem:main_largecontrol} with the initial state $w_{\delta,0} \coloneq {S}_{T_{\operatorname{ctrl}} - \delta}(w_0,h,\xvec{U}_0)|_{t=T_{\operatorname{ctrl}} - \delta}$ and external force $h(\cdot,  T_{\operatorname{ctrl}} - \delta + \cdot)$, there exist control parameters $(\zeta_{l})_{l\in\{1,\dots,N\}} \subset \xLtwo((0, 1); \mathbb{R})$ and a number $\varkappa_0 > 0$ such that
		\begin{equation}\label{equation:wdelta}
			w_{\varkappa} = {S}_{\varkappa}\left(w_{\delta,0},h(\cdot, T_{\operatorname{ctrl}} - \delta + \cdot) + \varkappa^{-1}\widetilde{\eta}_{\zeta_1,\dots,\zeta_N}(\cdot, \varkappa^{-1} \cdot), \widetilde{\xsym{\aleph}}_{\varkappa}\right)
		\end{equation}
		satisfies
		\begin{equation}\label{equation:prfvrtvstcpr}
			\|w_{\varkappa}(\cdot, \varkappa) - \widetilde{w}_1\|_{k} < \varepsilon/2
		\end{equation}
		for any $\varkappa \in (0, \varkappa_0)$. Because the convergence property \eqref{equation:deltatozero} in \Cref{theorem:main_largecontrol} is uniform with respect to initial data from a bounded subset of $\xHn{{k+1}}$ and external forces from a bounded subset of $\xLtwo((0,1);\xHn{{k-1}})$, keeping \eqref{equation:hfoba} in mind, we can take $\varkappa = \delta < \varkappa_0$ in \eqref{equation:wdelta}.
		
		As a result, the proof of \Cref{theorem:main0_general} can now be concluded by choosing $\sigma \coloneq \delta^{-1}$ and defining
		\begin{equation}\label{equation:defsoltctrlprob}
			w(\xvec{x},t) \coloneq \begin{cases}
				{S}_{T_{\operatorname{ctrl}}}(w_0,h,\xvec{U}_0)(\xvec{x},t), & (\xvec{x},t) \in \mathbb{T}^2 \times [0, T_{\sigma}],\\
				w_{\sigma^{-1}}(\xvec{x},t-T_{\sigma}), & (\xvec{x},t) \in \mathbb{T}^2 \times [T_{\sigma}, T_{\operatorname{ctrl}}],
			\end{cases}
		\end{equation}
		where $T_{\sigma} = T_{\operatorname{ctrl}}-\sigma^{-1}$. By \Cref{lemma:Wellposedness}, the function $w$ from \eqref{equation:defsoltctrlprob} is the unique solution to \eqref{equation:00NavierSTokesOrEuler} with the control $\eta = \eta_{\zeta_1,\dots,\zeta_N,\sigma}$, as specified in \eqref{equation:formeta}, and the velocity average $\xsym{\aleph}_{\sigma}$, as defined in \eqref{equation:Tsig}. In particular, owing to \eqref{equation:prfvrtvstcpr} with $\varkappa = \sigma^{-1}$, the function $w$ defined in \eqref{equation:defsoltctrlprob} obeys \eqref{equation:rcondthm}. Finally, the control parameters in \eqref{equation:formeta} are of the form
		\begin{equation}\label{equation:zetagam1}
			\gamma_l(t) \coloneq \sigma \mathbb{I}_{[T_{\sigma},T_{\operatorname{ctrl}}]}(t) \widetilde{\gamma}_l(\sigma (t-T_{\sigma})),
		\end{equation}
		where
		\begin{equation}\label{equation:zetagam2}
			\begin{aligned}\setstretch{0.3}
				\widetilde{\gamma}_l(t) \coloneq \begin{cases}
					\sum_{j=1}^K 
					\mathbb{I}_{[t^j_a,t^j_b]}(t) \zeta_{l}(\tau_j(t)), &  l \leq N,\\ & \\
					- \int_0^t \int_{\mathbb{T}^2} \widehat{\eta}_{(\zeta_{\xsym{\ell}}^s,\zeta_{\xsym{\ell}}^c)_{\xsym{\ell}\in\mathcal{K}}} (\xvec{z}, s) \, \xdx{\xvec{z}} \xdx{s}, & l = N+1,\\& \\
					- \int_{\mathbb{T}^2} \widehat{\eta}_{(\zeta_{\xsym{\ell}}^s,\zeta_{\xsym{\ell}}^c)_{\xsym{\ell}\in\mathcal{K}}} (\xvec{z}, t) \, \xdx{\xvec{z}} , & l = N + 2.
				\end{cases}
			\end{aligned}
		\end{equation}
		Moreover, inserting the formula for $\widehat{\eta}_{\zeta_1,\dots,\zeta_N}$ from \eqref{equation:ExplicitControl3} through \eqref{equation:zetagam2} into \eqref{equation:zetagam1} allows to express $\gamma_{N+1}$ and $\gamma_{N+2}$ in terms of $\gamma_{1}, \dots, \gamma_{N}$.
	\end{proof}
	\begin{rmrk}\label{remark:delta0delta}
		The proof of \Cref{theorem:main0_general} simplifies when allowing an additional control parameter. Namely, instead of Step~1 above, one could fix $\delta_0 > 0$ such that any trajectory to \eqref{equation:00NavierSTokesOrEuler} entering the $\varepsilon/2$-neighborhood of $w_1$ in $\xHn{k}$ with inactive control remains in the $\varepsilon$-neighborhood of $w_1$ in $\xHn{k}$ for time $\delta_0$. Then, in Step~2, one could glue suitable trajectories as follows. First, using no control until $t = T_{\operatorname{ctrl}} - \delta_0$, next steering in time $\delta \in (0, \delta_0)$ into the $\varepsilon/2$-neighborhood of $w_1$ in $\xHn{k}$, finally, switching again the control off during $[T_{\operatorname{ctrl}} - \delta_0 + \delta, T_{\operatorname{ctrl}}]$. However, a representation for such a control would require parameters $\gamma_1,\dots,\gamma_{N+2}$, $\sigma = \delta^{-1}$, and $\delta_0$.
	\end{rmrk}
	\begin{figure}[ht!]
		\centering
		\resizebox{0.9\textwidth}{!}{
			\begin{tikzpicture}
				\clip(0,-0.6) rectangle (8.8,5.6);
				
				\draw[line width=0.8mm, dashed, color=FireBrick!100] plot[smooth cycle] (4,3.5) circle (1.6);
				
				\coordinate[label=right:\scriptsize \color{FireBrick}$w_0$] (A) at (0.2,0.5);
				\coordinate[label=right:\scriptsize $\widetilde{w}_0$] (A) at (6.9,1.6);
				\draw [line width=0.4mm, dashed, ->, shorten <= 0.25cm, shorten >= 0.25cm] (0.8,0.5) to[out=20,in=-50] (6.9,1.5);
				
				\coordinate[label=right:\scriptsize \color{FireBrick}$w_1$] (B) at (3.68,3.2);
				\coordinate[label=right:\scriptsize$\widetilde{w}_1$] (B) at (3.15,4.1);
				\begin{scope}[transparency group, opacity=0.4]
					\draw [line width=0.4mm, ->, shorten <= 0.25cm, shorten >= 0.25cm] (7,1.42) to[out=-160,in=200] (3.5,3.8);
				\end{scope}
				\coordinate[label=right:\scriptsize$w_{\delta,0}$] (F) at (7.3,0.75);
				\draw [line width=0.4mm, ->, shorten <= 0.25cm, shorten >= 0.25cm] (7.3,0.8) to[out=-170,in=260] (2.85,3.3);
				
				\draw[line width=0.4mm, dashed, color=FireBrick!100] plot (4,3.5) -- (4.92,4.8);
				\coordinate[label=right:\color{FireBrick}\scriptsize$\varepsilon$] (G) at (4.4,4);

				\fill[line width=0mm, color=FireBrick!100, fill=FireBrick!100] plot[smooth cycle] (0.85,0.52) circle (0.12);
				\fill[line width=0mm, color=FireBrick!100, fill=FireBrick!100] plot[smooth cycle] (4,3.5) circle (0.12);
				\fill[line width=0mm, color=black, fill=black, opacity=0.4] plot[smooth cycle] (6.95,1.4) circle (0.1);
				\fill[line width=0mm, color=black, fill=black] plot[smooth cycle] (3.45,3.8) circle (0.1);
				\fill[line width=0mm, color=black, fill=black] plot[smooth cycle] (7.3,0.8) circle (0.1);
				
			\end{tikzpicture}
		}
		\caption{A sketch of several ideas from the proof of \Cref{theorem:main0_general}. Starting at $t = 0$ with the original initial state $w_0$, we first follow the uncontrolled trajectory (dashed line with arrow) until reaching a state $w_{\delta,0}$ that is close to the natural terminal point $\widetilde{w}_0$. Hereto, a good value of $\delta > 0$  is chosen such that the system can be steered in time $\delta$ from the state $\widetilde{w}_0$ to a small neighborhood of the regularized target $\widetilde{w}_1$. The value of $\delta$ is taken so that a controlled trajectory starting at $w_{\delta,0}$ can reach the desired small neighborhood of $\widetilde{w}_1$ in the same time $\delta$.}
		\label{Figure:activinactivecontr}
	\end{figure}
	\subsection{Conclusion of the main theorem}\label{subsection:prfmainthm}
	\Cref{theorem:main0_vel} is now a consequence of \Cref{theorem:main0_general}. To this end, we denote by $\varepsilon > 0$ the approximation accuracy parameter selected in \Cref{theorem:main0_vel}.
	
	\paragraph{Step 1. Controlling the vorticity.}
	We apply \Cref{theorem:main0_general} with the initial data $w_0 \coloneq \xwcurl{\xvec{u}_0}$,  target state $w_1 \coloneq  \xwcurl{\xvec{u}_1}$, prescribed force $h \coloneq \xwcurl{\xvec{f}}$, and accuracy parameter $\varepsilon > 0$ from \Cref{theorem:main0_vel}.
	This provides $\sigma_0 \geq T_{\operatorname{ctrl}}^{-1}$ and the existence of controls $\gamma_1,\dots\gamma_N \in \xLtwo((0, T_{\operatorname{ctrl}}); \mathbb{R})$
	such that the unique solution~$w$ to \eqref{equation:00NavierSTokesOrEuler}, with $\eta = \eta_{\gamma_1,\dots\gamma_N,\sigma}$ given by \eqref{equation:formeta},
	satisfies 
	\begin{equation}\label{equation:prfmttc}
		\|w(\cdot, T_{\operatorname{ctrl}}) - w_1\|_{k} < \frac{\varepsilon}{C_0}
	\end{equation}
	for any $\sigma \geq \sigma_0$, and where $C_0 > 0$ is the constant determined by \eqref{equation:velocityvorticityestimate}.
	\paragraph{Step 2. Integrating the vorticity control.}
	Due to the location of~$\operatorname{supp}(\chi)$ defined in \eqref{equation:Definition_chi}, we can choose a point $\xvec{p}^{\omegaup} = [p_{1}^{\omegaup},p_{2}^{\omegaup}]^{\top} \in \mathbb{T}^2$, a small number $d^{\omegaup} > 0$, and a length parameter $L^{\omegaup} > 0$ such that the square $\mathcal{O}^{\omegaup} \coloneq \xvec{p}^{\omegaup} + [0, L^{\omegaup}]^2$
	satisfies
	\begin{equation}\label{equation:sqp}
		\operatorname{supp}(\eta) \subset  \mathcal{O}^{\omegaup} \times (0, T_{\operatorname{ctrl}}) \subset \omegaup \times (0, T_{\operatorname{ctrl}}), \quad \operatorname{dist}(\mathcal{O}^{\omegaup}, \partial \omegaup) > d^{\omegaup},
	\end{equation}
	where $\partial \omegaup$ is the boundary of the control region $\omegaup \subset \mathbb{T}^2$.
	Moreover, we denote the auxiliary functions
	\begin{gather*}
		a(x_1,x_2,t) \coloneq \int_{p^{\omegaup}_1}^{x_1} \eta(s,x_2,t) \, \xdx{s}, \quad b(x_2, t) \coloneq \int_{p^{\omegaup}_2}^{x_2} a(p^{\omegaup}_1+L^{\omegaup}, s, t) \, \xdx{s},\\
		c(x_1,x_2,t) \coloneq a(x_1,x_2,t) -\rhoup(x_1)a(p^{\omegaup}_1+L^{\omegaup},x_2, t),
	\end{gather*}
	where $\rhoup \in \xCinfty(\mathbb{T};\mathbb{R})$ satisfies
	\[
		\operatorname{supp}(\rhoup) \subset (p^{\omegaup}_1,p^{\omegaup}_1+L^{\omegaup}+d^{\omegaup}), \quad \rhoup(s) = 1 \mbox{ for } s \in (p^{\omegaup}_1 + L^{\omegaup},p^{\omegaup}_1+L^{\omegaup} + d^{\omegaup}/2).
	\] 
	Then, inspired by \cite[Appendix A.2]{CoronMarbachSueur2020}, we define the vector field $\widehat{\xsym{\xi}} = [\widehat{\xi}_1,\widehat{\xi}_2]^{\top}\in \xLtwo((0,T_{\operatorname{ctrl}}); \xCinfty(\mathbb{T}^2; \mathbb{R}^2))$ by setting
	\begin{equation}\label{equation:ExplicitVelocityControl1}
		\begin{gathered}
			\widehat{\xi}_1(x_1,x_2,t) \coloneq \begin{cases}
				-\rhoup'(x_1)b(x_2, t),  & x_1 \in [0, p^{\omegaup}_1 + L^{\omegaup} + d^{\omegaup}/4 ),\\
				0, & x_1 \in [p^{\omegaup}_1 + L^{\omegaup} + d^{\omegaup}/4, 2\pi ),
			\end{cases} \\
			\widehat{\xi}_2(x_1,x_2,t) \coloneq \begin{cases}
				c(x_1,x_2,t), &  x_1 \in [p^{\omegaup}_1, p^{\omegaup}_1 + L^{\omegaup} + d^{\omegaup}/4],\\
				0, & \mbox{otherwise}
			\end{cases}
		\end{gathered}
	\end{equation}
	for $(x_1,x_2,t) \in \mathbb{T}^2\times(0,T_{\operatorname{ctrl}})$. Owing to this construction, recalling \eqref{equation:sqp} and the average-free property $\smallint_{\mathbb{T}^2} \eta(\xvec{x}, t) \, \xdx{\xvec{x}} = 0$, it follows that
	\[	
		\operatorname{supp}(\widehat{\xsym{\xi}}(\cdot,t)) \subset \omegaup, \quad \xwcurl(\widehat{\xsym{\xi}}(\cdot,t)) = \eta(\cdot,t).
	\]
	Finally, after obtaining explicit formulas for the profiles $\xsym{\vartheta}_1,\dots,\xsym{\vartheta}_{N+2}$ by inserting the expression for~$\eta$ from \eqref{equation:formeta} into \eqref{equation:ExplicitVelocityControl1}, one arrives at the representation
	\begin{equation}\label{equation:exprfc}
		\widehat{\xsym{\xi}}(\xvec{x},t) = \sum_{l=1}^N \gamma_l(t) \xsym{\vartheta}_l(\xvec{x}, 1 - \sigma (T_{\operatorname{ctrl}}-t)) + \sum_{l=N+1}^{N+2} \gamma_{l}(t) \xsym{\vartheta}_{l}(\xvec{x}),
	\end{equation}
	where $(\gamma_j)_{j \in \{1,\dots,N+2\}}$ are those from \eqref{equation:zetagam1}.
	Up to composing $\xsym{\vartheta}_{1}(\xvec{x}, \cdot), \dots, \xsym{\vartheta}_{N}(\xvec{x}, \cdot)$ with the transformation $t \mapsto 1 - t$, one can also write \eqref{equation:exprfc} in the form \eqref{equation:IntroductionLocalizedControl}, replacing $\xsym{\vartheta}_l(\xvec{x}, 1 - \sigma (T_{\operatorname{ctrl}}-t))$ by $\xsym{\vartheta}_l(\xvec{x},\sigma (T_{\operatorname{ctrl}}-t))$ without renaming the profiles.
	
	\paragraph{Step 3. Controlling the velocity average.} To describe the curl-free part of the force~$\xsym{\xi}$ in \eqref{equation:IntroductionLocalizedControl2}, let us explain how the building blocks~$\xsym{\vartheta}_{N+3}, \xsym{\vartheta}_{N+4} \in \xCinfty(\mathbb{T}^2; \mathbb{R}^2)$ are constructed with
	\begin{equation}\label{equation:propcutoffvf1}
		\begin{gathered}
			\xwcurl{\xsym{\vartheta}_{N+3}} = \xwcurl{\xsym{\vartheta}_{N+4}} = 0, \quad \operatorname{supp}(\xsym{\vartheta}_{N+3}) \cup \operatorname{supp}(\xsym{\vartheta}_{N+4}) \subset \Omega
		\end{gathered}
	\end{equation}
	and
	\begin{equation}\label{equation:propcutoffvf2}
	\begin{gathered}
		\mathbb{R}^2 = \operatorname{span}_{\mathbb{R}}\left\{\int_{\mathbb{T}^2} \xsym{\vartheta}_{N+3}(\xvec{x}) \, \xdx{\xvec{x}}, \int_{\mathbb{T}^2} \xsym{\vartheta}_{N+4}(\xvec{x}) \, \xdx{\xvec{x}}\right\}.
	\end{gathered}
\end{equation}
	This choice is enabled by the assumption that $\Omega$ contains two cuts $\mathcal{C}_1, \mathcal{C}_2 \subset \Omega$ with the property that $\mathbb{T}^2 \setminus (\mathcal{C}_1 \cup \mathcal{C}_2)$ is simply-connected  (\cf~\Cref{Figure:twocontroldomains}). For simplicity, let us here assume that $\mathcal{C}_1$ and $\mathcal{C}_2$ can be chosen as the graphs of smooth real-valued functions over the vertical and horizontal axis respectively, noting that the arguments for the general case are provided in \Cref{theorem:cutoffvectorfield}. Then, one can explicitly construct two functions $\upsilonup_1, \upsilonup_2 \colon \mathbb{T} \longrightarrow \mathbb{R}$ such that
	\[
		x_1 + \upsilonup_1(x_2) = 0 \iff (x_1,x_2) \in \mathcal{C}_1, \quad x_2 + \upsilonup_2(x_1) = 0 \iff (x_1, x_2) \in \mathcal{C}_2.
	\]
	Furthermore, we fix a small number $l > 0$ and any non-negative cutoff $\beta \in \xCinfty(\mathbb{T};\mathbb{R}_+)$ which obeys
	\[
		\forall i \in \{1,2\} \colon \, \operatorname{dist}(\mathcal{C}_i, \partial \Omega) < l, \quad \operatorname{supp}(\beta) \subset (-l/2, l/2), \quad  \beta(0) > 0.
	\]
	Then, we define the curl-free vector fields
	\begin{equation*}\label{equation:cutoffvfgof}
		\xsym{\vartheta}_{N+3} \coloneq \begin{bmatrix}
			\beta(x_1 + \upsilonup_1(x_2)) \\ \beta(x_1 + \upsilonup_1(x_2))\upsilonup_1'(x_2)
		\end{bmatrix}, \quad \xsym{\vartheta}_{N+4} \coloneq \begin{bmatrix}
			\beta(x_2 + \upsilonup_2(x_1))\upsilonup_2'(x_1)\\ \beta(x_2 + \upsilonup_2(x_1))
		\end{bmatrix},
	\end{equation*}
	which have the properties \eqref{equation:propcutoffvf1} and \eqref{equation:propcutoffvf2}. In particular, one can check that~$\xsym{\vartheta}_{N+3}$ and~$\xsym{\vartheta}_{N+4}$ have linearly independent averages by means of the calculation
	\begin{equation*}
		\begin{aligned}
			\int_{\mathbb{T}^2}\beta(x_i + \upsilonup_i(x_j))\partial_j\upsilonup_i(x_j) \, \xdx{\xvec{x}} & = - \int_{\mathbb{T}^2} \upsilonup_i(x_j)\upsilonup_i'(x_j) \beta'(x_i + \upsilonup_i(x_j)) \, \xdx{\xvec{x}}\\
			& =  - \int_{\mathbb{T}^2} \partial_i\left(\upsilonup_i(x_j)\upsilonup_i'(x_j) \beta(x_i + \upsilonup_i(x_j)) \right) \, \xdx{\xvec{x}} \\
			& = 0
		\end{aligned}
	\end{equation*}
	for $i, j \in \{1, 2\}, i \neq  j$. Now, we can choose the controls $\gamma_{N+3}$ and $\gamma_{N+4}$, depending on all the data of the controllability problem, by way of
	\begin{equation}\label{equation:gamma78}
		\gamma_{N+3}(t) = \partial_t A_1(t) - B_1(t), \quad \gamma_{N+4}(t) = \partial_t A_2(t) - B_2(t),
	\end{equation}
	where $A_1, A_2, B_1, B_2$ are chosen such that
	\begin{equation*}\label{equation:averagecoeff}
		\begin{gathered}
			\xsym{\aleph}_{\sigma}(t) = A_1(t) \int_{\mathbb{T}^2} \xsym{\vartheta}_{N+3}(\xvec{x}) \, \xdx{\xvec{x}} + A_2(t) \int_{\mathbb{T}^2} \xsym{\vartheta}_{N+4}(\xvec{x}) \, \xdx{\xvec{x}}, \\
			\int_{\mathbb{T}^2} \left(\widehat{\xsym{\xi}}(\xvec{x},t) + \xvec{f}(\xvec{x},t)\right) \, \xdx{\xvec{x}} = B_1(t) \int_{\mathbb{T}^2} \xsym{\vartheta}_{N+3}(\xvec{x}) \, \xdx{\xvec{x}} + B_2(t) \int_{\mathbb{T}^2} \xsym{\vartheta}_{N+4}(\xvec{x}) \, \xdx{\xvec{x}},
		\end{gathered}
	\end{equation*}
	with $\xsym{\aleph}_{\sigma}(t)$ as defined in \eqref{equation:Tsig} and $\widehat{\xsym{\xi}}$ from \eqref{equation:exprfc}.
	By inserting \eqref{equation:Tsig} and \eqref{equation:exprfc} into~\Cref{equation:gamma78}, and since $\gamma_1,\dots,\gamma_{N+2}$ are given in terms of $\gamma_1,\dots,\gamma_N$ as remarked below \eqref{equation:zetagam2}, one can also describe $\gamma_{N+3}$ and $\gamma_{N+4}$ through explicit formulas in terms of $\gamma_1,\dots,\gamma_N, \sigma$.
	\paragraph{Step 4. Conclusion.}
	Let $\xvec{u} \in \xCzero([0,T_{\operatorname{ctrl}}];\bVn{{k+1}})\cap\xLtwo((0,T_{\operatorname{ctrl}});\bVn{{k+2}})$ be the unique solution to \eqref{equation:NavierStokesVelocity} associated with the control force $\xsym{\xi} = \xsym{\xi}_{\gamma_1,\dots,\gamma_N,\sigma}$ of the form~\eqref{equation:IntroductionLocalizedControl2}. By the above constructions, the profiles $\xsym{\vartheta}_{1}, \dots, \xsym{\vartheta}_{N+4}$ appearing in the representation of~$\xsym{\xi}$ are universal, as they only depend on the choice of $\Omega$ (and the subset $\omegaup$) fixed in \Cref{section:introduction}. Further, we notice that, by integrating the velocity equation of the Navier--Stokes problem \eqref{equation:NavierStokesVelocity}, one has
	\[
		\int_{\mathbb{T}^{2}} \xvec{u}(\xvec{x},t) \, \xdx{\xvec{x}} = \int_{\mathbb{T}^{2}} \xvec{u}(\xvec{x},0) \, \xdx{\xvec{x}} + \int_0^t \left( \int_{\mathbb{T}^{2}} \xsym{\xi}(\xvec{x},s) \, \xdx{\xvec{x}} + \int_{\mathbb{T}^{2}} \xvec{f}(\xvec{x},s) \, \xdx{\xvec{x}} \right) \, \xdx{s}.
	\]
	As a result, taking into account \eqref{equation:gamma78} and the estimates \eqref{equation:velocityvorticityestimate} and \eqref{equation:prfmttc}, $\xvec{u}$ is seen to satisfy the terminal condition
	\[
		\|\xvec{u}(\cdot, T_{\operatorname{ctrl}}) - \xvec{u}_1\|_{k+1} \leq C_0	\|w(\cdot, T_{\operatorname{ctrl}}) - w_1\|_{k} < \varepsilon,
	\]
	which completes the proof of \Cref{theorem:main0_vel}.

	\begin{xmpl}\label{example:ControlWithStandardBasisAsGenerator}
		A key aspect of this article is the possibility to choose $N = 4$ in \Cref{theorem:main0_vel}. To this end, one could take
		$\mathcal{K} \coloneq \{ [1,0]^{\top}, [0,1]^{\top}\}$ and observe that
		\[
			\left\{s_{\xsym{\ell}}, c_{\xsym{\ell}}\right\}_{\xsym{\ell} \in \mathcal{K}} = \left\{\xvec{x}\mapsto\sin(x_1), \xvec{x}\mapsto\cos(x_1), \xvec{x}\mapsto\sin(x_2), \xvec{x}\mapsto\cos(x_2)\right\}.
		\]
	\end{xmpl}

	\paragraph{Acknowledgment.} The present work was initiated during a visit of VN to Shanghai Jiao Tong University (SJTU) in January 2022, where MR was a Ph.D. student at that time. The authors would like to extend their gratitude towards SJTU and Professors Ya-Guang Wang and Siran Li for their hospitality, as well as for providing an excellent research environment. We further would like to thank the anonymous referees for their constructive remarks.

	\addcontentsline{toc}{section}{References}
	
	\bibliographystyle{alpha}
	\bibliography{BibLLC}
	
	\appendix
	
	\pagebreak
	\section{Curl-free vector fields supported near smooth cuts}
	The following theorem is likely known. As we could not locate the desired statement in the literature, the details are here provided in form of explicit constructions. 
	\begin{thrm}\label{theorem:cutoffvectorfield}
		Let $\Omega \subset \mathbb{T}^2$ be a subdomain containing smooth cuts $\mathcal{C}_1, \mathcal{C}_2 \subset \Omega$ such that $\mathbb{T}^2\setminus (\mathcal{C}_1 \cup \mathcal{C}_2)$ is simply-connected. There exist $\xsym{\Lambda}, \xsym{\Sigma} \in \xCinfty(\mathbb{T}^2; \mathbb{R}^2)$ satisfying
		\begin{equation*}
			\mathbb{R}^2 = \operatorname{span}_{\mathbb{R}}\left\{\int_{\mathbb{T}^2} \xsym{\Lambda} \, \xdx{\xvec{x}}, \int_{\mathbb{T}^2} \xsym{\Sigma} \, \xdx{\xvec{x}}\right\}, \,\,  \xwcurl{\xsym{\Lambda}} = \xwcurl{\xsym{\Sigma}} = 0, \,\, \operatorname{supp}(\xsym{\Lambda}) \cup \operatorname{supp}(\xsym{\Sigma}) \subset \Omega.
		\end{equation*}
	\end{thrm}

	\begin{figure}[ht!]
		\centering
		\resizebox{1\textwidth}{!}{
			\begin{subfigure}[b]{0.5\textwidth}
				\centering
				\resizebox{1\textwidth}{!}{
					\begin{tikzpicture}
						\clip(-0.1,-0.1) rectangle (9.1,9.1);
						
						\draw[line width=0.5mm, color=black] plot [smooth] coordinates {(1,0) (1,0.5) (1, 1) (1.5, 1.5) (2, 2) (2.5, 2.2) (3,3) (3.2,2.8) (3.5, 2.5) (3.7, 2.3) (4, 2) (3.5, 1.5) (3.7, 1.3) (4, 1) (4.2, 0.8)  (4.5, 0.5) (5,1) (6,3) (6.5,3.5) (7,4) (8,5) (7.5,5.5) (7,6) (6.5,6.5) (5,8) (2.5,4) (2,4.5) (1.5,5) (1,6) (1,7) (1,9)};
						
						\draw[line width=1.5mm, color=FireBrick!60] (1.5,1.5) -- (2,2);
						\draw[line width=1.5mm, color=FireBrick!60] (3.2,2.8) -- (3.7, 2.3);
						\draw[line width=1.5mm, color=FireBrick!60] (3.7, 1.3) -- (4.2, 0.8);
						\draw[line width=1.5mm, color=FireBrick!60] (6.5,3.5) -- (7,4);
						\draw[line width=1.5mm, color=FireBrick!60] (7,6) -- (6.5,6.5);
						\draw[line width=1.5mm, color=FireBrick!60] (2,4.5) -- (1.5,5);

						\draw[line width=0.2mm, color=black, dashed, dash pattern=on 8pt off 8pt] plot[smooth cycle] (0,0) rectangle ++(9,9);
						
					\end{tikzpicture}
				}
				\label{Figure:controldomain3}
			\end{subfigure}
			\hfill \quad \hfill
			\begin{subfigure}[b]{0.5\textwidth}
				\centering
				\resizebox{1\textwidth}{!}{
					\begin{tikzpicture}
						\clip(-0.1,-0.1) rectangle (9.1,9.1);
						
						\draw[line width=0.5mm, color=black] plot [smooth] coordinates {(1,0) (1,0.5) (1, 1) (1.5, 1.5) (2, 2) (2.5, 2.2) (3,3) (3.2,2.8) (3.5, 2.5) (3.7, 2.3) (4, 2) (3.5, 1.5) (3.7, 1.3) (4, 1) (4.2, 0.8)  (4.5, 0.5) (5,1) (6,3) (6.5,3.5) (7,4) (8,5) (7.5,5.5) (7,6) (6.5,6.5) (5,8) (2.5,4) (2,4.5) (1.5,5) (1,6) (1,7) (1,9)};
						
						\draw[line width=1.5mm, color=FireBrick!60] (1.5,1.5) -- (2,2);
						\draw[line width=1.5mm, color=FireBrick!60] (3.2,2.8) -- (3.7, 2.3);
						\draw[line width=1.5mm, color=FireBrick!60] (3.7, 1.3) -- (4.2, 0.8);
						\draw[line width=1.5mm, color=FireBrick!60] (6.5,3.5) -- (7,4);
						\draw[line width=1.5mm, color=FireBrick!60] (7,6) -- (6.5,6.5);
						\draw[line width=1.5mm, color=FireBrick!60] (2,4.5) -- (1.5,5);
						
						\draw[line width=0.6mm, color=DarkGray!140, dashed] (1.5,1.95) -- (1.86,1.59);
						\draw[line width=0.6mm, color=DarkGray!140, dashed] (3.28,2.38) -- (3.7, 2.8);
						\draw[line width=0.6mm, color=DarkGray!140, dashed] (3.73, 0.83) -- (4.1, 1.2);
						\draw[line width=0.6mm, color=DarkGray!140, dashed] (6.5,4) -- (7,3.5);
						\draw[line width=0.6mm, color=DarkGray!140, dashed] (6.58,6.08) -- (6.95,6.45);
						\draw[line width=0.6mm, color=DarkGray!140, dashed] (1.58,4.58) -- (1.96,4.96);
						
						\draw[line width=0.6mm, color=DarkGray!140, dashed] plot [smooth] coordinates {(0.8,0) (0.8,0.5) (0.8, 1) (1.2, 1.6) (2, 2.5) (2.5, 2.7) (3, 3.4) (4.2,2.1) (3.7,1.55) (4,1.3) (4.5,0.92) (5,1.6) (5.5,3) (6.1,3.5) (6.5,4) (7.5,5) (7,5.5) (6.65,6) (6.2,6.4) (5,7.5) (2.5,3.5) (1.8,4.4) (1.1,5) (0.8,6) (0.8,7) (0.8,9)};
						
						\draw[line width=0.6mm, color=DarkGray!140, dashed] plot [smooth] coordinates {(1.2,0) (1.2,0.5) (1.2, 0.7) (1.4, 1.2) (2, 1.7) (2.5, 1.9) (3, 2.5) (3.75, 1.95) (3.2, 1.5) (3.6, 1) (4.4, 0.2) (5,0.6) (6.4,2.8) (6.98,3.5) (7.5,4) (8.4,5) (7.5,5.9) (7,6.4) (6.5,6.9) (5,8.4) (2.5,4.7) (2,4.9) (1.4,5.6) (1.2,6.4) (1.2,7) (1.2,9)};
						
						\draw[line width=0.2mm, color=black, dashed, dash pattern=on 8pt off 8pt]	plot[smooth cycle] (0,0) rectangle ++(9,9);
						
						\coordinate[label=right:{\color{black}\large$\xS_1$}] (B) at (1.25,0.3);
						\coordinate[label=right:{\color{black}\large$\xS_1$}] (B) at (1.25,8.7);
						\coordinate[label=right:{\color{black}\large$\xS_2$}] (B) at (3.95,8.22);
						\coordinate[label=right:{\color{black}\large$\xS_3$}] (B) at (6.65,5);
						\coordinate[label=right:{\color{black}\large$\xS_4$}] (B) at (6,2.15);
						\coordinate[label=right:{\color{black}\large$\xS_5$}] (B) at (3.8,2.7);
						\coordinate[label=right:{\color{black}\large$\xS_6$}] (B) at (1.7,2.9);
					\end{tikzpicture}
				}
				\label{Figure:controldomain4}
			\end{subfigure}
		}
		\caption{An illustration of the cut $\mathcal{C}_1$. The displayed tubular neighborhood of $\mathcal{C}_1$ comprises six sections $\xS_1, \dots, \xS_6$ such that $\mathcal{C}_1$ is for $k \in \{1,2,3\}$ a graph over the $x_1$-axis in $\xS_{2k}$ and a graph over the $x_2$-axis in $\xS_{2k-1}$. Near the internal section boundaries, the curve has the slope $1$ or $-1$.}
		\label{Figure:twocontroldomains2}
	\end{figure}

	\begin{proof}
		Since $\Omega$ is open, we may re-choose the smooth cuts $(\mathcal{C}_i)_{i \in \{1,2\}}$ in a way that $\mathbb{T}^2\setminus (\mathcal{C}_1 \cup \mathcal{C}_2)$ is simply-connected and the following properties hold.
		\begin{itemize}
			\item The cut $\mathcal{C}_1$ equals the straight line $x_1 = c_1$ in $\mathbb{T}\times ([0,r_1]\setminus(r_1/3,2r_1/3))$, for some constant $c_1 \in \mathbb{T}$ and small $r_1 > 0$. Similarly, the cut $\mathcal{C}_2$ equals the line $x_2 = c_2$ in $([0,r_2]\setminus(r_2/3,2r_2/3))\times\mathbb{T}$ for some $c_2 \in \mathbb{T}$ and small $r_2 > 0$.
			\item There exists a tubular neighborhood $\xN(\mathcal{C}_i) = \bigcup_{k=1}^{l_i}\xS_{k}(\mathcal{C}_i)$ of $\mathcal{C}_i$ with disjoint sections $\xS_1(\mathcal{C}_i)$, \dots $\xS_{l_i}(\mathcal{C}_i)$, in a way that $\mathcal{C}_i$ is in each $\xS_{2k+1-i}(\mathcal{C}_i)$ a graph over the $x_1$-axis and in $\xS_{2k-2+i}(\mathcal{C}_i)$ a graph over the $x_2$-axis.
			\item The intersection $\xR_{l,k}(\mathcal{C}_i) \coloneq \partial \xS_l(\mathcal{C}_i) \cap \partial \xS_k(\mathcal{C}_i)$ of two adjacent sections $\xS_l(\mathcal{C}_i)$ and $\xS_{k}(\mathcal{C}_i)$ is a single line segment with slope either $1$ or $-1$. In the vicinity of the square with diagonal $\xR_{l,k}(\mathcal{C}_i)$, the curve $\mathcal{C}_i$ equals a straight line segment $\mathcal{L}_{l,k}(\mathcal{C}_i) \subset \mathcal{C}_i$, with slope either $-1$ or $1$, such that the line $\xR_{l,k}(\mathcal{C}_i)$ passes transversely through $\mathcal{L}_{l,k}(\mathcal{C}_i)$.
		\end{itemize}
		For the sake of a concise presentation, we assume that $\mathcal{C}_1$ is the curve displayed in \Cref{Figure:twocontroldomains2} and only construct the vector field $\xsym{\Lambda}$. The vector field $\xsym{\Sigma}$ can be built in the same manner. In particular, that $\xsym{\Lambda}$ and $\xsym{\Sigma}$ have linearly independent averages will turn out as a generic property. Our treatment of the example in \Cref{Figure:twocontroldomains2} provides all the building blocks required for considering any general region $\Omega$ meeting the hypotheses of \Cref{theorem:cutoffvectorfield}. 
		
		\paragraph{Step 1. Constructions.} Since we consider here only the curve $\mathcal{C}_1$, let us fix for each $k \in \{1,\dots, l_1 = 6\}$ the names
		\[
			\xN = \xN(\mathcal{C}_1), \quad \xS_k = \xS_k(\mathcal{C}_1).
		\]
		Then, by further reducing the diameter of the tube $\xN$ if necessary, we select six smooth functions $\widetilde{\upsilonup}_k, \widehat{\upsilonup}_k, \colon \mathbb{T} \longrightarrow \mathbb{R}_-$, $k \in \{1,2,3\}$, having the below listed attributes.
		\begin{itemize}
			\item In a neighborhood $\xN_{2k-1}$ of the section $\xS_{2k-1}$, satisfying $\overline{\xS_{2k-1}} \subset \xN_{2k-1}$, it holds that $x_1 + \widetilde{\upsilonup}_{k}(x_2) = 0$ if and only if $(x_1,x_2) \in \mathcal{C}_1$. When $(x_1,x_2) \in \xN_{2k-1}$ is located right to $\mathcal{C}_1$, then $x_1 + \widetilde{\upsilonup}_k(x_2) > 0$. If $(x_1,x_2) \in \xN_{2k-1}$ is located left to $\mathcal{C}_1$, then $x_1 + \widetilde{\upsilonup}_k(x_2) < 0$.
			\item In a neighborhood $\xN_{2k}$ of section $\xS_{2k}$, satisfying $\overline{S}_{2k} \subset \xN_{2k}$, it holds that $x_2 + \widehat{\upsilonup}_k(x_1) = 0$ if and only if $(x_1,x_2) \in \mathcal{C}_1$. When $(x_1,x_2) \in \xN_{2k}$ is located above $\mathcal{C}_1$, then $x_2 + \widehat{\upsilonup}_k(x_1) > 0$. If $(x_1,x_2) \in \xN_{2k}$ is located below $\mathcal{C}_1$, then $x_2 + \widehat{\upsilonup}_k(x_1) < 0$.
		\end{itemize}
		Furthermore, for a sufficiently small number $l \in (0, \operatorname{dist}(\mathcal{C}_1, \partial \Omega))$, which will be determined later, we choose a cutoff $\beta \in \xCinfty(\mathbb{T};\mathbb{R}_+)$ obeying
		\begin{equation}\label{equation:appcutoffprop}
			\operatorname{supp}(\beta) \subset (-l/2, l/2), \quad  \beta(0) > 0.
		\end{equation}
		Let us introduce the four main building blocks. Namely, for any general function~$\upsilonup\colon \mathbb{T} \longrightarrow \mathbb{R}_-$ and $\xvec{x} = [x_1,x_2]^{\top} \in \mathbb{T}^2$, we define
		\[
			\widetilde{\xsym{\Lambda}}^{\upsilonup,\pm}(\xvec{x}) \coloneq
			\begin{bmatrix}
				\pm\beta(\pm x_1 \pm \upsilonup(x_2)) \\ \pm\beta(\pm x_1 \pm \upsilonup(x_2)) \upsilonup'(x_2)
			\end{bmatrix}, \,\, \widehat{\xsym{\Lambda}}^{\upsilonup,\pm}(\xvec{x}) \coloneq
			\begin{bmatrix}
			\pm\beta(\pm x_2 \pm \upsilonup(x_1)) \upsilonup'(x_1) \\ \pm\beta(\pm x_2 \pm \upsilonup(x_1))
			\end{bmatrix}.
		\]
		Then, we build $\xsym{\Lambda}$ by gluing the previously introduced functions in a suitable way, resulting in the explicit formula
		\begin{equation*}\label{equation:lambdageneral}
			\xsym{\Lambda}(x_1,x_2) \coloneq \begin{cases}
				\widetilde{\xsym{\Lambda}}_k(x_1,x_2) & \mbox{ if } (x_1,x_2) \in \xS_{2k-1},\\
				\widehat{\xsym{\Lambda}}_k(x_1,x_2) & \mbox{ if } (x_1,x_2) \in \xS_{2k},\\
				\xsym{0} & \mbox{ otherwise,}
			\end{cases}
		\end{equation*}
		where the smooth vector fields
		\[
			(\widetilde{\xsym{\Lambda}}_k \colon \xN_{2k-1} \longrightarrow \mathbb{T}^2)_{k\in\{1,2,3\}}, \quad (\widehat{\xsym{\Lambda}}_k \colon \xN_{2k} \longrightarrow \mathbb{T}^2)_{k\in\{1,2,3\}}
		\]
		are given by
		\begin{align*}
				\widetilde{\xsym{\Lambda}}_1(\xvec{x}) & = \widetilde{\xsym{\Lambda}}^{\widetilde{\upsilonup}_1,+}, && \widetilde{\xsym{\Lambda}}_2(\xvec{x})  = \widetilde{\xsym{\Lambda}}^{\widetilde{\upsilonup}_2,+}, && \widetilde{\xsym{\Lambda}}_3(\xvec{x})  = \widetilde{\xsym{\Lambda}}^{\widetilde{\upsilonup}_3,-},\\
				\widehat{\xsym{\Lambda}}_1(\xvec{x}) & = \widehat{\xsym{\Lambda}}^{\widehat{\upsilonup}_1,+}, && \widehat{\xsym{\Lambda}}_2(\xvec{x})  = \widehat{\xsym{\Lambda}}^{\widehat{\upsilonup}_2,-}, && \widehat{\xsym{\Lambda}}_3(\xvec{x})  = \widehat{\xsym{\Lambda}}^{\widehat{\upsilonup}_3,-}.
		\end{align*}
		The small parameter $l > 0$ is fixed in a way that $\widetilde{\xsym{\Lambda}}_k$ and $\widehat{\xsym{\Lambda}}_k$ are for each $k \in \{1,2,3\}$ supported in a neighborhood of $\mathcal{C}_1$ which is sufficiently thin to ensure that $\xsym{\Lambda}$ is a well-defined, smooth, and curl-free vector field obeying  $\operatorname{supp}(\xsym{\Lambda}) \subset \Omega$.
		\paragraph{Step 2. Checking the average.} It remains to study the average of $\xsym{\Lambda} = [\Lambda_1, \Lambda_2]^{\top}$. To this end, we write $\xU \coloneq \mathbb{T}\times(0, r_1)$ and decompose
		\[
				\int_{\mathbb{T}^2} \xsym{\Lambda}(\xvec{x}) \, \xdx{\xvec{x}} = 	\int_{\xU} \xsym{\Lambda}(\xvec{x}) \, \xdx{\xvec{x}} + \int_{\mathbb{T}^2 \setminus \xU} \xsym{\Lambda}(\xvec{x}) \, \xdx{\xvec{x}}.
		\]
		Due to \eqref{equation:appcutoffprop} and Fubini's theorem, or alternatively by virtue of the Divergence Theorem, we arrive at
		\begin{equation*}
			\begin{aligned}
				\int_{\xU} \Lambda_1(\xvec{x}) \, \xdx{\xvec{x}} > 0, \quad \int_{\xU} \Lambda_2(\xvec{x}) \, \xdx{\xvec{x}} = 0.
			\end{aligned}
		\end{equation*}
		Therefore, by slightly perturbing $\mathcal{C}_1$ in $\mathbb{T} \times (r_1/3,2r_1/3)$, one can change the value of $\smallint_{\mathbb{T}^2} \Lambda_1(\xvec{x}) \, \xdx{\xvec{x}}$ without affecting $\smallint_{\mathbb{T}^2} \Lambda_2(\xvec{x}) \, \xdx{\xvec{x}}$. The vector field $\xsym{\Sigma} = [\Sigma_1, \Sigma_2]^{\top}$ is then obtained by analogous constructions, but now along the smooth cut $\mathcal{C}_2$. Therefore, one can modify $\smallint_{\mathbb{T}^2} \Sigma_2(\xvec{x}) \, \xdx{\xvec{x}}$ in $(r_2/3, 2r_2/3) \times \mathbb{T}$ without impacting the value of the integral $\smallint_{\mathbb{T}^2} \Sigma_1(\xvec{x}) \, \xdx{\xvec{x}}$. In conclusion, we can first construct candidates for $\xsym{\Lambda}$ and $\xsym{\Sigma}$, followed by performing slight perturbations, if necessary, so that their averages are rendered linearly independent. 
	\end{proof}

\end{document}